\documentclass[oneside,english]{amsart}
\usepackage{libertineRoman}
\usepackage{libertineMono}
\usepackage[T3,T1]{fontenc}
\usepackage{pifont}
\usepackage{mathtools}
\usepackage{enumitem}
\usepackage{bm}
\usepackage{amsrefs}
\usepackage{amsbsy}
\usepackage{amstext}
\usepackage{amssymb}
\usepackage{setspace}
\usepackage{xargs}[2008/03/08]
\usepackage{tikz}
\usepackage{relsize}
\usepackage[unicode=true,pdfusetitle,
 bookmarks=true,bookmarksnumbered=true, breaklinks=false, pdfborder={0 0 0}]
 {hyperref}
\makeatletter

\AtBeginDocument{

	}

\@namedef{subjclassname@2020}{%
	\textup{2020} Mathematics Subject Classification}

\DeclareSymbolFont{tipa}{T3}{cmr}{m}{n}
\DeclareMathAccent{\everb}{\mathalpha}{tipa}{16} 
\makeatother

\newtheorem{theorem}{Theorem}[section]
\newtheorem*{theorem*}{Theorem}
\newtheorem{claim}{Claim}[theorem]
\newtheorem{lemma}[theorem]{Lemma}
\newtheorem{fact}{Fact}

\newtheorem{proposition}[theorem]{Proposition}
\newtheorem{corollary}[theorem]{Corollary}

\theoremstyle{definition}
\newtheorem*{definition*}{Definition}

\newtheorem{definition}[theorem]{Definition}

\theoremstyle{remark}
\newtheorem{remark}[theorem]{Remark}

\newtheorem{question}[theorem]{Question}

\newcommand{\nocontentsline}[3]{}
\newcommand{\notoc}[2]{\bgroup\let\addcontentsline=\nocontentsline#1{#2}\egroup}

\makeatletter

\providecommand{\LyX}{\texorpdfstring%
	{L\kern-.1667em\lower.25em\hbox{Y}\kern-.125emX\@}
	{LyX}}

\makeatother

\DeclareMathOperator{\otp}{otp}

\DeclareMathOperator{\dom}{dom}

\DeclareMathOperator{\sk}{sk}

\makeatletter
\def\moverlay{\mathpalette\mov@rlay}
\def\mov@rlay#1#2{\leavevmode\vtop{%
		\baselineskip\z@skip \lineskiplimit-\maxdimen
		\ialign{\hfil$\m@th#1##$\hfil\cr#2\crcr}}}
\newcommand{\charfusion}[3][\mathord]{
	#1{\ifx#1\mathop\vphantom{#2}\fi
			\mathpalette\mov@rlay{#2\cr#3}
		}
	\ifx#1\mathop\expandafter\displaylimits\fi}
\makeatother

\newcommand\zfc{\mathrm{ZFC}}
\newcommand\zf{\mathrm{ZF}}

\renewcommand\mid{\mathrel{|}\allowbreak}

\newcommand{\mets}{\mathord{\upharpoonright}}
\newcommand{\power}{\mathcal{P}}

\newcommand{\cof}[1]{\mathrm{cf}\left(#1\right)}
\newcommand{\cpow}[1]{\mathcal{P}_{\omega_{1}}(#1)}
\newcommand{\emp}{\varnothing}
\newcommand{\fii}{\varphi}
\newcommand{\till}{, \dots ,}
\newcommand{\smin}{\smallsetminus}
\newcommand\T{\mathfrak{T}}%
\newcommand\p{\mathbb{P}}%
\newcommand\qp{\mathbb{Q}}%
\newcommand\LL{\mathcal{L}}%
\newcommand\G{\mathfrak{G}}%

\newcommand{\psu}[2]{\prescript{#1}{}{#2}}%
\newcommand\ord{\mathrm{Ord}}%
\newcommand\hod{\mathrm{HOD}}%

\newcommand\lex{\mathrm{Lex}}%

\DeclareMathOperator{\supp}{supp}

\newcommand\con{\subseteq}%
\newcommand\spo{\vartriangleleft}%
\newcommand\pdwn{\mathord{\downarrow}}%
\newcommand\pup{\mathord{\uparrow}}%
\newcommand\der{\partial}%
\newcommand\Der{\bm{\partial}}%
\newcommand\bigast{\mathop{\mathlarger{\mathlarger{*}}}}%
\newcommand{\fin}[1]{[#1]^{<\omega}}
\newcommand{\meq}[1]{#1/{\sim}}
\newcommandx\sing[1][usedefault, addprefix=\global, 1=]{{\bf Sing}_{#1}}%

\newcommand\des[1]{\mathbf{Des}(#1)}%
\newcommand\nr[1]{\mathbf{NR}(#1)}%

\DeclareMathOperator{\Def}{Def}

\DeclareMathOperator{\aaa}{\mathtt{aa}}
\DeclareMathOperator{\stat}{\mathtt{stat}}

\begin{document}

\title{ITERATED CLUB SHOOTING AND THE STATIONARY LOGIC CONSTRUCTIBLE MODEL}
\author{UR YA'AR}
\address{Einstein Institute of Mathematics %
	Hebrew University of Jerusalem \newline%
	Current address: Department of Mathematics and Statistics, University of Helsinki}
\email{ur.yaar@proton.me}

\begin{abstract}
	We investigate iterating the construction of $C(\aaa)$, the $L$-like
	inner model constructed using stationary logic.
	We show that it is possible to force over generic extensions of $L$
	to obtain a model of $V=C(\aaa)$, and to obtain models in which the
	sequence of iterated $C(\aaa)$s is decreasing of arbitrarily large
	order-types. For this we prove distributivity and stationary-set preservation
	properties for countable iterations of club-shooting forcings using
	mutually stationary sets,
	and introduce the notion of mutually fat sets which yields better distributivity results even for uncountable iterations.
\end{abstract}



\maketitle
\section{Introduction and preliminaries}

\subsection{Introduction}

The model $C(\aaa)$, introduced by Kennedy, Magidor and V\"{a}\"{a}n\"{a}nen
in \cite{IMEL2}, is the model of sets constructible using
\emph{stationary logic} $\LL(\aaa)$ -- first order logic augmented
with the quantifiers $\aaa$ and $\stat$,
meaning roughly ``for club/stationarily many countable subsets'' (see definition \ref{def:Stationary-logic} below).
This is a model of $\zf$, and one can phrase the
formula ``$V=C(\aaa)$'', i.e. $\forall x\exists\alpha(x\in L'_{\alpha})$
where $L'_{\alpha}$ is the $\alpha$th level in the construction
of $C(\aaa)$ (see definition \ref{def:Caa}). However, it is not always
true that $C(\aaa)\vDash``V=C(\aaa)"$,
which is equivalent to the question whether $C(\aaa)^{C(\aaa)}=C(\aaa)$. This is clearly the
case if $V=L$, so the interesting question is whether this can hold
with $C(\aaa)\ne L$. In section \ref{sec:coding} we show that this
is consistent relative to the consistency of $\zfc$. Next we investigate
the possibilities of $C(\aaa)\nvDash ``V=C(\aaa)"$. In such a case, it
makes sense to define recursively the iterated $C(\aaa)$s:
\begin{align*}
	C(\aaa)^{0}        & =V                                                                               \\
	C(\aaa)^{\alpha+1} & =C(\aaa)^{C(\aaa)^{\alpha}}\,\text{ for any \ensuremath{\alpha}}                 \\
	C(\aaa)^{\alpha}   & =\bigcap_{\beta<\alpha}C(\aaa)^{\beta}\,\,\text{ for limit \ensuremath{\alpha}}.
\end{align*}
This type of construction was first investigated by McAloon \cite{mcaloon2}
regarding $\mathrm{HOD}$, where he showed that it is equiconsistent
with $\zfc$ that there is a strictly decreasing sequence of iterated
$\mathrm{HOD}$ of length $\omega$, and the intersection of the sequence
can be either a model of $\zfc$ or of $\zf+\mathrm{\neg AC}$. Harrington
also showed (in unpublished notes, cf. \cite{Zadrozny-IteratingOrdinalDefinability1983}) that the
intersection might not even be a model of $\zf$. Jech \cite{jech1975descending}
used forcing with Suslin trees to show that it is possible to have
a strictly decreasing sequence of iterated $\mathrm{HOD}$ of any
arbitrary ordinal length, and later Zadro\.{z}ny \cite{zadrozny1981transfinite}
improved this to an $\ord$ length sequence. In \cite{Zadrozny-IteratingOrdinalDefinability1983} Zadro\.{z}ny
generalized McAloon's method and gave a more flexible framework for
coding sets by forcing, which he used to give another proof of this
result. As $\hod$ can also be described as the model constructed
using second order logic (as shown by Myhill and Scott \cite{MyhillScott}),
it is natural to ask which of the results for $\hod$ apply to other
such models

Our goal in this paper is to use Zadro\.{z}ny's framework in the context
of $C(\aaa)$, and our main challenges would be in finding the appropriate
coding tools for this case. As we are dealing with stationary logic,
the natural candidates are club shooting forcings. We will show how
to use such forcings to code sets into the $C(\aaa)$ construction
by choosing exactly which stationary sets we destroy (out of
a predetermined list). The first challenge would be to find a way
to iterate such forcings without destroying what we've already coded.
The second challenge would be to investigate the limit stages of the
iterated $C(\aaa)$ sequence. After some preliminaries regarding stationary
sets and intersections of generic extensions, in section \ref{subsec:Shooting-countably-many}
we use the notion of mutually stationary sets to form countable iterations
of club shooting forcings, which will allow us, in sections \ref{subsec:Coding-into-V=00003DC(aa)}
and \ref{subsec:Iterating-C(aa)}, to obtain models of $V=C(\aaa)\ne L$,
and models with descending sequences of iterated $C(\aaa)$s of countable
length.

In order to obtain longer descending sequences of iterated $C(\aaa)$s,
in section \ref{subsec:Shooting-uncountably-many} we introduce the
notion of \emph{mutually fat sets}, and show that these kinds of sets
allows us to form arbitrarily large iterations of club shooting forcings.
We then provide two ways of obtaining mutually fat sets -- using
specific $\square$-sequences (section \ref{subsec:square-sequence})
and by forcing non-reflecting stationary sets (section \ref{subsec:Forcing-non-reflecting-stationar}).
We eventually use the first option, in Theorem \ref{thm:descending-long},
to obtain $C(\aaa)$-sequences of any predetermined order-type.

These results should be contrasted with the case of $C^{*}$ -- the
model constructed from the cofinality $\omega$ logic. $C^{*}$ is
contained in $C(\aaa)$, but in \cite{Yaar-IteratingCofinalityoConstructible2023} we show
that having a model with a descending sequence of iterated $C^{*}$
is equiconsistent with the existence of an inner model with a measurable
cardinal, while here we obtain such models and even more over $L$.
This further demonstrates the difference in expressive power between
stationary logic and the cofinality-$\omega$ quantifier.

\subsection{\label{subsec:Stationary-sets}Stationary sets}

In this paper we will have two notions of club and stationary sets
-- one regarding ordinals, and one regarding countable subsets
of some given set. In most cases it will be clear from the context
which notion is used, and otherwise we will state it explicitly.
We first recall the definitions.
\begin{definition}\label{def:club-and-stationary}
	For an ordinal $\alpha$, a subset $C\con\alpha$ is called \emph{closed} if it contains all its limit points and it is \emph{club} if it is closed and unbounded in $\alpha$.
	$C$ is \emph{$\sigma$-closed} if it contains all limit points of its countable subsets.
	$S\con\alpha$ is called \emph{stationary} if it intersects every club
	in $\alpha$.
	For regular $\lambda<\kappa$ we denote by $E_{\lambda}^{\kappa}$ the set
	$\left\{ \alpha<\kappa\mid\cof{\alpha}=\lambda\right\} $
	and similarly for $\lambda\leq\kappa$ $E_{<\lambda}^{\kappa}=\left\{ \alpha<\kappa\mid\cof{\alpha}<\lambda\right\} $.
	It is well known that these are stationary sets.

	For an arbitrary set $X$, $C\con\cpow X$ is called \emph{club in
		$X$  }if there is some algebra $\mathfrak{A}=\left\langle X,f_{n}\right\rangle _{n<\omega}$
	(where $f_{n}:X^{k_n}\to X$ for some $k_n<\omega$) such that
	\[
		C=C_{\mathfrak{A}}:=\left\{ z\in\power_{\omega_{1}}(X)\mid\forall n(f_{n}''z^{k_n}\con z)\right\}
	\]
	i.e. the collection of all subsets of $X$ closed under all functions
	of $\mathfrak{A}$.
	$S\con \cpow{X}$ is called \emph{stationary
		in $\cpow X$} if for every algebra $\mathfrak{A}$ on $X$, $S\cap C_{\mathfrak{A}}\ne\emp$.
\end{definition}

Note that club subsets form a filter on $\power_{\omega_{1}}(X)$,
and the stationary sets are the positive sets with respect to this
filter. Using Skolem functions, the club filter is also generated
by club sets which consist of all \emph{elementary substructures} of some structure on $X$.
And similarly -- a stationary set is one that contains an elementary substructure of any structure on $X$.
The notions of club and stationary sets can also be defined using
single functions $F:\left[X\right]^{<\omega}\to X$, where a club
is the set of all countable sets closed under such $F$ (denoted $C_F$) and a stationary set is
such that for each such $F$ there is a member closed under $F$.
A useful characterization, which often is taken as the definition (see e.g. \cite{KunenNew, Jech}), is that
$C$ is club in $\power_{\omega_{1}}(X)$
iff it is closed under countable unions of chains,
and unbounded in the sense that for every $a\in\cpow X$ there is $c\in C$ such that $a\con c$
(the equivalence of the definitions is due to Menas \cite{Menas-StrongCompactnessSupercompactness1975}; see also \cite{jech2009stationary}).
If we say that some set of countable sets $S$ is stationary, without
mentioning any ambient set, we mean that it is stationary in $\cpow{\cup S}$.

The following lemma connects the two notions. For $A\con\kappa$ and
$B\con\cpow{\kappa}$ denote
\begin{align*}
	\everb{A} & =\left\{ X\in\cpow{\kappa}\mid\sup X\in A\right\} \\
	\breve{B} & =\left\{ \sup X \mid X\in B\right\} .
\end{align*}
Notice that if $A \con E^\kappa_{\omega}$ then $\breve{\everb{A}}=A$, and $\everb{\breve{B}}\supseteq B$,
but in the second case there might not be equality, as there may be more subsets of $\kappa$ having the same suprema as those in $B$.
\begin{lemma}
	\label{lem:correspond}Let $\kappa>\omega$ be regular, $A\con E_{\omega}^{\kappa}$,
	$B\con\cpow{\kappa}$.
	\begin{enumerate}
		\item \label{enu:ordclub->club}  $A$ is $\sigma$-closed and unbounded in $\kappa$ $\implies$$\everb{A}$
			      contains a club in $\cpow{\kappa}$.
		\item \label{enu:club->ordclub}  $\everb{A}$ is club in $\cpow{\kappa}$ $\implies$ $A$ is $\sigma$-closed
		      and unbounded in $\kappa$.
		\item \label{enu:stat->ordstat} $B$ is stationary in $\cpow{\kappa}$ $\implies$ $\breve{B}$ is
		      stationary in $\kappa$.
		\item \label{enu:corres-stat}$A$ is stationary in $\kappa$ $\iff$
		      $\everb{A}$ is stationary in $\cpow{\kappa}$.
	\end{enumerate}
\end{lemma}

\begin{proof}
	\noindent		\ref{enu:ordclub->club}.	If $A$ is $\sigma$-closed and unbounded,
	then $\everb{A}$ contains the set of all $X\in\cpow{\kappa}$ closed under the function $\alpha\mapsto\min A\smin\alpha$.

	\noindent		\ref{enu:club->ordclub}. If $\everb{A}$ is club  in $\cpow{\kappa}$
	then using the last characterization we gave, unboundedness of $\everb{A}$ clearly gives the unboundedness of $A$.
	Now let $\left<\alpha_n \mid n<\omega \right>$ be an increasing sequence of elements of $A$.
	For every  $n$ let $X_n \in \everb{A}$ such that $\sup X_n = \alpha_n$.
	Note that $\sup (X_n \cup X_{n+1})=\alpha_{n+1}$ so also $X_n \cup X_{n+1}\in \everb{A}$,
	hence we can assume without loss of generality that $X_n \con X_{n+1}$.
	Hence by closure of $\everb{A}$ under unions of chains, also $\bigcup_{n<\omega}X_n \in A$, which by definition means $\sup_{n<\omega} \alpha_n \in A$.

	\noindent\ref{enu:stat->ordstat}. 	Now assume $B$ is stationary in $\cpow{\kappa}$. If $C\con\kappa$
	is club, then by clause \ref{enu:ordclub->club} $\everb{C}$ contains a club in $\cpow{\kappa}$
	so there is some $X\in\everb{C}\cap B$, and $\sup X \in C\cap\breve{B}$.
	So $\breve{B}$ is stationary in $\kappa$.

	\noindent		\ref{enu:corres-stat}. 	Assume now $A$ is stationary in $\kappa$, and let $C\con\cpow{\kappa}$ be club.
	Let $F:[\kappa]^{<\omega}\to \kappa$ be such that $C=C_{F}$.
	The set of $\alpha<\kappa$  such that $F''[\alpha]^{<\omega} \con \alpha$ is club in $\kappa$, so there is such  $\alpha \in A$.
	As $A \con E^\kappa_\omega$, $\cof \alpha = \omega$, so let $X \con \alpha$ be a countable cofinal sequence.
	Then $\sup F''[X]^{<\omega} = \alpha \in A$ hence $F''[X]^{<\omega} \in C \cap \everb{A}$.
	$C$ was an arbitrary club, hence $\everb{A}$ is stationary in $\cpow{\kappa}$.
	On the other hand,
	if $\everb{A}$ is stationary in $\cpow{\kappa}$ then we've already
	shown that $A=\breve{\everb{A}}$ is stationary in $\kappa$.
\end{proof}

An important property of stationary sets in $\cpow{\kappa}$ is that
they project upwards and downwards, i.e. they form a \emph{tower}:
\begin{lemma}
	\label{lem:projections} Let $Y\supseteq X\ne\emp$.
	\begin{enumerate}
		\item If $S\con\power_{\omega_{1}}\left(Y\right)$ is stationary then $S\pdwn X:=\left\{ Z\cap X\mid Z\in S\right\} $
		      is stationary.
		\item If $S\con\power_{\omega_{1}}\left(X\right)$ is stationary then $S\pup Y:=\left\{ Z\in\power_{\omega_{1}}\left(Y\right)\mid Z\cap X\in S\right\} $
		      is stationary.
	\end{enumerate}
\end{lemma}

\begin{proof}
	1. Let $H:\left[X\right]^{<\omega}\to X$, define $\bar{H}:\left[Y\right]^{<\omega}\to Y$
	by $\bar{H}\left(\bar{a}\right)=H\left(\bar{a}\cap X\right)$, and
	let $\bar{Z}\in S$ closed under $\bar{H}$. Then  $Z\cap X\in S\pdwn X$
	is closed under $H$.

	\noindent 2. Let $G:\left[Y\right]^{<\omega}\to Y$. Given the variables $\left\langle x_{i}\mid i<\omega\right\rangle $,
	enumerate all terms in $G$ as $\left\langle t_{i}\mid i<\omega\right\rangle $
	such that $t_{i}$ is a term in the variables $x_{0}\till x_{i-1}$
	(usually there will be dummy variables), and wrap them all in one
	function $G':\left[Y\right]^{<\omega}\to Y$ where $G'\left(\bar{a}\right)$
	is $t_{i}\left(a_{0}\till a_{i-1}\right)$ for $i=\left|\bar{a}\right|$.
	Let $x\in X$ and define $G'':[X]^{<\omega} \to X$ by
	\[
		G''\left(\bar{a}\right)= \begin{cases}
			G'\left(\bar{a}\right) & \text{if } G'\left(\bar{a}\right)\in X \\
			x                      & \text{otherwise}
		\end{cases}
	\]
	Let $Z'\in S$ be closed under $G''$,
	and $Z$ the closure of $Z'$ under $G$.
	$Z'$ is countable so its closure is also countable.
	$Z\cap X=Z'$ since every $a\in Z\smin Z'$,
	is of the form $G'\left(\bar{b}\right)$
	for some $\bar{b}\in\fin {Z'}$, but if $G'\left(\bar{b}\right) \in X$, then by closure under	$G''$, it would be in $Z$.
	So $Z'\in S\pup Y$ and it is closed under $G$.
\end{proof}
The following lemmas are basic tools in preserving stationarity under
forcing:
\begin{definition}
	\label{def:generic-condition}Let $\p$ be a forcing notion, $p\in\p$
	and $a$ a countable set such that $\p,p\in a$. We say that $p'\in\p$
	is \emph{generic over} $a$ (or $a$-generic in short) if for every
	dense open set $D\con\p$ such that $D\in a$, $p'\in D$. $p'$ is
	generic over $a$ \emph{below $p$ }if in addition $p'\leq p$.
\end{definition}

\begin{lemma}
	\label{lem:stat-preserve} Let $\p$ be a forcing notion and $T\con\cpow{\kappa}$
	stationary. Let $\theta\geq\kappa$ be large so that $\p\in H(\theta)$.
	If for every $a\prec H(\theta)$ with $\p\in a$ such that $a\cap\kappa\in T$
	and every $p\in a$ there is an $a$-generic below $p$, then $\p\Vdash T$
	is stationary.
\end{lemma}

\begin{proof}
	Let $p\in\p$ and $\dot{F}$ a name for a function from $ [\kappa]^{<\omega}$
	to $\kappa$. Let $\mathfrak{A}=\langle H(\theta),\p,p,\dot{F}\rangle$
	for large enough $\theta$.
	By stationarity of $T$ and Lemma \ref{lem:projections}, there is $a\prec\mathfrak{A}$ such that $a\cap\kappa\in T$.
	By the assumption there is $p'\leq p$ $a$-generic.
	For every $x\in [a\cap\kappa]^{<\omega}$ there is in $a$ a dense subset of $\p$ determining the value of $\dot{F}(x)$ to be some member of $a$.
	$p'$ meets this set, so there is some $y\in a\cap\kappa$ such that $p'\Vdash\dot{F}(x)=y$.
	So $p'$ forces that $a\cap\kappa\in T$ is closed under $\dot{F}$.
	Hence for every $p\in\p$ there is $p'\leq p$ forcing that there is an element of $T$ closed under $\dot{F}$.
	$\dot{F}$ was arbitrary, so indeed $\p$ forces that $T$ is stationary.
\end{proof}
\begin{lemma}
	\label{lem:sigma-closed}If $\p$ is a $\sigma$-closed forcing and
	$T\con\cpow{\kappa}$ is stationary then $\p\Vdash T$ is stationary.
\end{lemma}

\begin{proof}
	Let $\theta\geq\kappa$ be large so that $\p\in H(\theta)$. Let $a\prec H(\theta)$
	with $\p\in a$ such that $a\cap\kappa\in T$ and $p\in a$. We show
	there is a generic over $a$ below $p$, which will be enough by the
	previous lemma. Since $a$ is countable, we can enumerate all its
	dense open sets $\left\langle D_{n}\mid n<\omega\right\rangle $,
	and inductively define a decreasing sequence of conditions in $\p\cap a$
	where $p_{0}=p$ and $p_{n+1}$ is chosen from $D_{n}$ below $p_{n}$
	(possible by denseness). By $\sigma$-closure there is $p'$ such
	that for every $n$, $p'\leq p_{n}$ and by openness, $p'\in D_{n}$
	for every $n$. So $p'$ is generic over $a$ below $p$.
\end{proof}

\subsection{\label{subsec:Stationary-logic}Stationary logic and its constructible
	model}
Stationary logic was introduced by Shelah in \cite{Shelah-GeneralizedQuantifiersCompact1975}, and first studied by Barwise, Kaufmann and Makkai in \cite{StatLogic}.
\begin{definition}
	\label{def:Stationary-logic}\emph{Stationary logic}, denoted $\LL(\aaa)$ ($\aaa$ stands for  ``almost all''),
	is the extension of first-order logic obtained by adding to the syntax:
	\begin{itemize}
		\item Second-order variables, usually denoted $s_i,t_j$ etc.
		\item Second-order quantifiers $\aaa$ and  $\stat$, each binding a single second-order variable.
	\end{itemize}
	As for semantics, second order variables are assumed to range over countable subsets.
	That is, the interpretation of an atomic formula of the form $\varphi=s(x)$, where $x$ is a first order variable and  $s$ a second order variable, in a model $\mathcal{M}$ with domain $M$, is given by assigning some $a\in M$ to $x$, some  $A\in \cpow M$ to  $s$, and letting
	\[
		\mathcal{M} \vDash_{\LL(\aaa)} \varphi(A,a) \iff a \in A.
	\]
	We might abuse notation and write $x \in s$ instead of  $s(x)$.
	%
	Logical connectives and first-order quantifiers are treated in the usual way, and the new quantifiers are given the following semantics:
	\begin{eqnarray*}
		\mathcal{M}\vDash_{\LL(\aaa)}\mathtt{aa}s\fii\left(s,\boldsymbol{t},\boldsymbol{a}\right)\iff & \mathclap{\underset{{\textstyle \text{contains a club in \ensuremath{\cpow{\ensuremath{M}}}}}}{\left\{ A\in\cpow M\mid\mathcal{M}\vDash_{\LL(\aaa)} \fii\left(A,\boldsymbol{t},\boldsymbol{a}\right)\right\} }}\\
		\mathcal{M\vDash_{\LL(\aaa)} \mathtt{stat}}s\fii\left(s,\boldsymbol{t},\boldsymbol{a}\right)\iff & \underset{{\textstyle \text{is stationary in \ensuremath{\cpow{\ensuremath{M}}}}}}{\left\{ A\in\cpow M\mid\mathcal{M}\vDash_{\LL(\aaa)} \fii\left(A,\boldsymbol{t},\boldsymbol{a}\right)\right\} }
	\end{eqnarray*}
	where $\boldsymbol{a}$ is a finite sequence of elements of $M$ and
	$\boldsymbol{t}$ a finite sequence of countable subsets of $M$.
	Note that
	\[
		\vDash_{\LL(\aaa)} \mathtt{stat}s \varphi(s, \boldsymbol{t}, \boldsymbol{a}) \leftrightarrow \neg\mathtt{aa}s\neg\fii\left(s,\boldsymbol{t},\boldsymbol{a}\right).
	\]
	We will usually drop the subscript from $\vDash_{\LL(\aaa)}$ if there's no risk of confusion.
\end{definition}

Following their general framework for models constructed from extended
logics of \cite{IMEL}, Kennedy Magidor and V\"{a}\"{a}n\"{a}nen introduce
the model constructed from stationary logic in \cite{IMEL2}:
\begin{definition}
	\label{def:Caa}$C\left(\aaa\right)$ is defined recursively by:
	\begin{align*}
		L'_{0}        & =\emp                                                                       \\
		L'_{\alpha+1} & =\Def_{\LL(\aaa)}\left(L'_{\alpha}\right)                                   \\
		L_{\beta}'    & =\bigcup_{\alpha<\beta}L_{\alpha}'\,\,\,\text{for limit \ensuremath{\beta}} \\
		C(\aaa)       & =\bigcup_{\alpha\in \ord}L'_{\alpha}
	\end{align*}
	where
	\[
		\Def_{\LL(\aaa)}(M)=\Big\{ \big\{ a\in M\mid(M,\in)\vDash_{\LL(\aaa)}\fii(a,\boldsymbol{b})\big\} \mid\fii\in\LL(\aaa),\,\boldsymbol{b}\in M^{<\omega}\Big\} .
	\]
\end{definition}

$C(\aaa)$ is a model of $\zf$, but recent results of Kennedy, Magidor
and V\"{a}\"{a}n\"{a}nen (see the appendix of \cite{IMEL2}) suggest it might not always be a model of the Axiom of Choice ($\mathrm{AC}$).
However in this paper it will always turn out to be a forcing extension
of $L$, so it will satisfy $\mathrm{AC}$.
\footnote{Due to this issue, in  \cite{IMEL2} the authors revise the definition of $C(\aaa)$, and the model defined above is denoted by $C_o(\aaa)$ ($o$ for  $o$ld). However to avoid excessive notation we will stick with the original notation here. It is still open whether there is an actual distinction between the resulting models, but recent research indicates that in the cases discussed here there is no distinction.}
\begin{definition}
	\label{def:iterated-Caa}The sequence of iterated $C(\aaa)$s is defined
	recursively by:
	\begin{align*}
		C(\aaa)^{0}        & =V                                                                               \\
		C(\aaa)^{\alpha+1} & =C(\aaa)^{C(\aaa)^{\alpha}}\,\text{ for any \ensuremath{\alpha}}                 \\
		C(\aaa)^{\alpha}   & =\bigcap_{\beta<\alpha}C(\aaa)^{\beta}\,\,\text{ for limit \ensuremath{\alpha}}.
	\end{align*}
	This will be our main object of study in this paper.
\end{definition}

\subsection{\label{subsec:Intersections-of-forcing}Intersections of forcing
	extensions}

In order to investigate the limit stages of the iterated $C(\aaa)$
construction, we need to understand intersections of generic extensions.
The basic facts are the following:
\begin{fact}
	\label{fact:descending-generic}Let $B$ be a complete Boolean algebra,
	\[
		B_{0}\supseteq B_{1}\supseteq\dots\supseteq B_{\alpha}\supseteq\dots\,\,\left(\alpha<\kappa\right)
	\]
	a descending sequence of complete subalgebras of $B$, $B_{\kappa}=\bigcap_{\alpha<\kappa}B_{\alpha}$,
	$G$ a $V$-generic filter on $B$ and for every $\alpha\leq\kappa$,
	$G_{\alpha}=G\cap B_{\alpha}$. Then
	\begin{enumerate}
		\item (Grigorieff \cite{Grigorieff75}) $\bigcap_{\alpha<\kappa}V\left[G_{\alpha}\right]$
		      satisfies $\zf$.
		\item (Jech \cite[lemma 26.6]{Jech1978}) If $B$ is $\kappa$-distributive
		      then $\bigcap_{\alpha<\kappa}V\left[G_{\alpha}\right]=V[G_{\kappa}]$,
		      and in particular satisfies $\zfc$ .
	\end{enumerate}
\end{fact}

So one of our  challenges would be to obtain the distributivity of
the forcing notions we wish to use. To get a more precise result,
we will use a characterization by Sakarovitch \cite{sakarovitch,sakarovitch-note}
giving an exact form to $V[G_{\kappa}]$ in the $\kappa$-distributive
case.
In the following, all forcing notions are assumed to be separative.
\begin{definition}
	\label{def:normal-emb}Let $\p,\qp$ be forcing notions. A function
	$f:\p\to\qp$ is called \emph{normal} iff it is order preserving, $f''\p$
	is dense in $\qp$ and
	\[
		\forall p\in\p\forall q\in\qp\left(q\leq f(p)\to\exists p'\in\p\left(p'\leq p\land f(p')\leq q\right)\right).
	\]

	A collection $\left\langle \p_{\alpha},f_{\alpha\beta}\mid\alpha,\beta<\kappa\right\rangle $
	is called a $\kappa$\emph{-normal system} if for every $\alpha<\beta<\kappa$,
	$f_{\alpha\beta}:\p_{\alpha}\to\p_{\beta}$ is normal, and $\alpha<\beta<\gamma\to f_{\alpha\gamma}=f_{\beta\gamma}\circ f_{\alpha\beta}$.
\end{definition}

We denote by $\mathrm{ro}(\p)$ the Boolean completion of the poset $\p$, and we consider  $\p$ to be a subset of  $\mathrm{ro}(\p)$.


\begin{proposition}\label{prop:Boolean-subalg-id}
	Let $\p,\qp$ be forcing notions and  $f:\p \to \qp$ a normal function.
	Then there is a complete embedding $h: \qp \to \mathrm{ro}(\p)$ (i.e. $h$ preserves order, incompatibility and maximality of antichains).
\end{proposition}
\begin{proof}
	Define $h(q)=\sum \{ p \in \p \mid f(p) \leq q \} $, where $\sum$ denotes the least upper bound in the complete Boolean algebra $\mathrm{ro}(\p)$.
	\begin{itemize}
		\item \textbf{Order preservation:} Assume $q' \leq q$, then for every $p \in \p$, $f(p) \leq q'$ implies $f(p) \leq q$, hence $h(q') \leq h(q)$.
		\item \textbf{Incompatibility preservation:} Assume $q \bot q'$, we need to show that $h(q)\bot h(q')$.
		      In a complete Boolean algebra we have
		      \begin{align*}
			      h(q)\cdot h(q') & = \sum \{ p \in \p \mid f(p) \leq q \} \cdot \sum \{ p \in \p \mid f(p) \leq q' \} \\
			                      & = \sum \{ p \cdot p' \mid f(p) \leq q, f(p') \leq q' \}
		      \end{align*}
		      and $h(q) \bot h(q')$ iff this is $0$ iff for all appropriate  $p,p'$, $p \cdot p' = 0$.
		      But if this is not the case, then there is (by density of $\p$ in  $\mathrm{ro}(\p)$) $s \in \p$ below some such $p,p'$, but then $f(s) \leq q,q'$ by contradiction.
		\item \textbf{Maximality preservation:} Assume $A \con \qp$ is a maximal antichain and let $b \in \mathrm{ro}(\p)$.
		      Let $p\in \p$ such that  $p\leq b$ and consider $f(p)$.
		      There is $q \in A$ compatible with $f(p)$, so let $q'$ be a common extension.
		      $q' \leq f(p)$ so by normality there is $p' \leq p$ such that $f(p') \leq q'$.
		      So $p' \leq h(q') \leq h(q)$ and also $p' \leq b$ so $b$ is compatible with an element of  $h''A$. \qedhere
	\end{itemize}
\end{proof}
This means that we can identify $\qp$, and then also $\mathrm{ro}(\qp)$, with subsets of $\mathrm{ro(\p)}$, and that in general a $\kappa$-normal system of posets corresponds, after the appropriate identifications, to a descending sequence of complete Boolean algebras.

\begin{definition}
	\label{def:eqrel}Let $\left\langle \p_{\alpha},f_{\alpha\beta}\mid\alpha,\beta<\kappa\right\rangle $
	be a $\kappa$\emph{-normal system}. We define an equivalence relation
	$\sim$ on $\p_{0}$ by $p\sim q$ iff $\exists\alpha<\kappa(f_{0,\alpha}(p)=f_{0,\alpha}(q))$,
	let $\meq{\p_{0}}$ be the set of equivalence classes $\left\{ [p]\mid p\in\p_{0}\right\} $,
	and denote $f_{0,\kappa}(p)=[p]$. We give $\meq{\p_{0}}$ the natural
	order:
	\[
		[p]\leq[q]\iff\exists\alpha<\kappa\forall\beta<\kappa(\beta\geq\alpha\to f_{0,\beta}(p)\leq f_{0,\beta}(q)).
	\]
	Note that this doesn't depend on the choice of representatives so
	could be stated for some/all $p'\in[p]$ and $q'\in[q]$. Then
	we have the following fact (\cite{sakarovitch}, see also \cite{sakarovitch-note}
	and \cite[section 8.1]{Zadrozny-IteratingOrdinalDefinability1983}):
	\begin{fact}
		\label{fact:Sakarovitch}Let $\left\langle \p_{\alpha},f_{\alpha\beta}\mid\alpha,\beta<\kappa\right\rangle $
		be a $\kappa$\emph{-normal system}, for every $\alpha$ let $B_{\alpha}=\mathrm{ro}(\p_{\alpha})$
		and $B_{\kappa}=\bigcap_{\alpha<\kappa}B_{\alpha}$. Then the \textbf{separative
			quotient} of \ $\meq{\p_{0}}$ is isomorphic to a dense subset of $B_{\kappa}$.

		If $G$ is $\p_{0}$ generic over $V$, then the set $f_{0,\kappa}''G$,
		denoted by $\meq G$, is $\meq{\p_{0}}$-generic over $V$, and hence
		$V[G_{\kappa}]=V[\meq G]$.
	\end{fact}

\end{definition}

\begin{remark}
	One should be warned that even if $\bigcap_{\alpha<\kappa}\p_{\alpha}$
	is trivial, this does not mean that $\bigcap_{\alpha<\kappa}B_{\alpha}$
	is trivial as well, and in general we cannot claim that $\bigcap_{\alpha<\kappa}B_{\alpha}=\mathrm{ro}\left(\bigcap_{\alpha<\kappa}\p_{\alpha}\right)$.
\end{remark}

We sketch a proof in the following case, from which the general case
can be derived:
\begin{proposition}
	Let $B$ be a complete Boolean algebra,
	\[
		B_{0}\supseteq B_{1}\supseteq\dots\supseteq B_{\alpha}\supseteq\dots\,\,\left(\alpha<\kappa\right)
	\]
	a descending sequence of complete subalgebras of $B$, $B_{\kappa}=\bigcap_{\alpha<\kappa}B_{\alpha}$.
	For $b,c\in B$ let $b\sim c$ iff $\exists\alpha<\kappa$ such that
	\[
		\inf\{d\in B_{\alpha}\mid d\geq b\}=\inf\{d\in B_{\alpha}\mid d\geq c\}.
	\]

	Then the  \emph{separative quotient} of $B_{0}/{\sim}$ is isomorphic
	to $B_{\kappa}$\emph{.}
\end{proposition}

\begin{samepage}
	\begin{proof}
		Recall that the separative quotient of a poset $\p$ is the unique
		(up to isomorphism) separative $\qp$ such that there is $h:\p \to \qp$
		surjective, order preserving, and satisfying that $x$ is compatible with $y$ iff $h(x)$
		is compatible with $h(y)$ (see \cite[p. 205]{Jech}). Since $B_{\kappa}$
		is separative as a Boolean algebra, we want to provide such $h:\meq{B_{0}}\to B_{\kappa}^{+}$.

		Let $b\in B_{0}^{+}$. For every $\alpha<\kappa$ let $b_{\alpha}=\inf\{d\in B_{\alpha}\mid d\geq b\}$,
		and let $\bar{b}=\sup\{b_{\alpha}\mid\alpha<\kappa\}$. Note that
		$\{b_{\alpha}\mid\alpha<\kappa\}$ is an ascending sequence, so in
		fact for every $\beta$, $\bar{b}=\sup\{b_{\alpha}\mid\beta\leq\alpha<\kappa\}$,
		and this is an element of $B_{\beta}$, so all-in-all $\bar{b}\in B_{\kappa}$.
		Now if $b'\sim b$ then for all large enough $\alpha$, $b_{\alpha}=b_{\alpha}'$
		so $\bar{b}=\bar{b'}$. So the function $h([b])=\bar{b}$ is well
		defined, and since in particular $\bar{b}\geq b>0$, it is into $B_{\kappa}^{+}$.

		We check that  $h$ is as required.
		\begin{itemize}
			\item \textbf{Surjectivity:} If $b \in B_\kappa$, then for every $\alpha$ $b_\alpha=b$, hence $\bar{b}=b$, so $h([b])=b$. Thus $h$ is surjective.
			\item \textbf{Order and compatibility preservation:}
			      If $[b],[c]\in\meq{B_{0}}$ are such that $[b] \leq [c]$, then for all large enough $\alpha$ $b_{\alpha} \leq c_{\alpha}$, so $\bar{b} \leq \bar{c}$. Hence $h$ is order preserving.
			      This also implies  that if $[b],[c]$ are compatible then also $\bar{b}, \bar{c}$ are compatible.
			\item \textbf{Incompatibility preservation:}
			      If $[b],[c]\in\meq{B_{0}}$ are incompatible, it means that there is an unbounded
			      $I\subseteq\kappa$ such that for $\alpha\in I$, $b_{\alpha}$ and
			      $c_{\alpha}$ are incompatible. But this also implies that for \emph{every}
			      $\alpha,\beta\in I$ $b_{\alpha}$ and $c_{\beta}$ are incompatible
			      (if $\alpha<\beta$ and there is e.g. $d\leq b_{\alpha},c_{\beta}$
			      then since $b_{\alpha}\leq b_{\beta}$ we get $d\leq b_{\beta},c_{\beta}$),
			      so, as we are in a complete Boolean algebra,
			      \[
				      \overline{b}\cdot\bar{c}=\sum_{\alpha\in I}b_{\alpha}\cdot\sum_{\beta\in I}c_{\beta}=\sum_{\alpha,\beta}b_{\alpha}\cdot c_{\beta}=0
			      \]
			      i.e. $\bar{b},\overline{c}$ are incompatible.
		\end{itemize}
	\end{proof}
\end{samepage}

Combining this with the previous fact (abusing notation a bit by using
$G_{\alpha}$ to denote both the generic in $\p_{\alpha}$ and in
$B_{\alpha}$) we get the following:
\begin{fact}
	\label{fact:Sak-dist}If, under the notation above, $\p_{0}$ is $\kappa$-distributive,
	then
	\[
		\bigcap_{\alpha<\kappa}V\left[G_{\alpha}\right]=V[G_{\kappa}]=V[\meq G].
	\]
\end{fact}

Lets consider now a particular case which will be of interest, where
the descending sequence corresponds to decreasing tails of a product.
\begin{lemma}
	\label{lem:kappa-closed}Let $\langle\qp_{\xi}^{\alpha}\mid\alpha\leq\xi<\kappa\rangle$
	be a collection of posets such that for every $\xi<\kappa$ there
	is a $\xi$-normal system $g_{\xi}^{\alpha\beta}:\qp_{\xi}^{\alpha}\to\qp_{\xi}^{\beta}$
	($\alpha<\beta\leq\xi)$. For $\alpha<\kappa$ let $\p_{\alpha}=\prod_{\alpha\leq\xi<\kappa}\qp_{\xi}^{\alpha}$
	(full support product). The normal systems for $\langle\qp_{\xi}^{\alpha}\mid\alpha\leq\xi<\kappa\rangle$
	induce a $\kappa$-normal system $f_{\alpha\beta}:\p_{\alpha}\to\p_{\beta}$
	where $f_{\alpha\beta}(p)(\xi)=g_{\xi}^{\alpha\beta}(p(\xi))$. If
	$\sim$ is the equivalence relation on $\p_{0}$ induced from this
	system, then $\meq{\p_{0}}$ is $\kappa$-closed.
\end{lemma}

\begin{remark}
	To simplify, one should have in mind the simple case where for every
	$\xi$, we have $\qp_{\xi}^{\alpha}=\qp_{\xi}$ for all $\alpha\leq\xi$,
	so we are just dealing with tails of the product: $\p_{\alpha}=\prod_{\alpha\leq\xi<\kappa}\qp_{\xi}$.
	And in the slightly less simple case, we have that for every $\xi$,
	$\langle\qp_{\xi}^{\alpha}\mid\alpha\leq\xi\rangle$ is a descending
	sequence of subposets such that the inclusion satisfies the normality
	condition.
\end{remark}

\begin{proof}
	Let $\left\langle p_{i}\mid i<\kappa\right\rangle $ be a sequence such
	that $i<j$ implies $[p_{i}]>[p_{j}]$. Let $\alpha_{0}=0$ and for
	every $i<\kappa$, if $\alpha_{i}$ is defined then fix $\alpha_{i+1}>\alpha_{i}$
	such that $\forall\beta\geq\alpha_{i+1}$ $f_{0\beta}(p_{i})\geq f_{0\beta}(p_{i+1})$.
	For limit $i$ set $\alpha_{i}=\sup_{j<i}\alpha_{j}$. Note that $i\leq\alpha_{i}$
	so $\sup_{i<\kappa}\alpha_{i}=\kappa$. By induction, using the fact
	that the sequence $\left\langle \alpha_{i}\mid i<\kappa\right\rangle $
	is strictly increasing, we get that for every $\beta\geq\alpha_{i+1}$
	and $m<n\leq i$ we have $f_{0\beta}(p_{m})\geq f_{0\beta}(p_{n})$.

	We now define $q\in\p_{0}$ by $q(\xi)=p_{i}(\xi)$ for the unique
	$i<\kappa$ such that $\xi\in[\alpha_{i},\alpha_{i+1})$. We claim
	that for every $i$ $[p_{i}]\geq[q]$, in particular that $\forall\beta\geq\alpha_{i+1}$,
	$f_{0\beta}(p_{i})\geq f_{0\beta}(q)$. So fix $\beta\geq\alpha_{i+1}$
	and $\xi\geq\beta$. Let $j$ be such that $\xi\in[\alpha_{j},\alpha_{j+1})$,
	so $q(\xi)=p_{j}(\xi)$. Note that $\alpha_{j+1}>\xi\geq\beta\geq\alpha_{i+1}$
	so $j>i$, so using the above claim we have $f_{0\beta}(p_{i})\geq f_{0\beta}(p_{j})$
	so:
	\[
		f_{0\beta}(p_{i})(\xi)\geq f_{0\beta}(p_{j})(\xi)=g_{\xi}^{0\beta}(p_{j}(\xi))=g_{\xi}^{0\beta}(q(\xi))=f_{0\beta}(q)(\xi)
	\]
	as required.
\end{proof}

\section{\label{sec:iterated-cs}Iterated club shooting}

\subsection{\label{subsec:Shooting-one-club}Shooting one club}

The basic tool for changing the notion of stationarity over some model
is ``club shooting'' -- adding a club to the complement of a stationary
set. For this to be possible, the complement must be ``fat'':
\begin{definition}
	\label{def:fat}Let $\kappa$ be a regular cardinal. A set $S\con\kappa$
	is called \emph{fat} iff for every club $C\con\kappa$, $S\cap C$
	contains closed sets of ordinals of arbitrarily large order-types
	below $\kappa$. We say that $T\con\kappa$ is \emph{co-fat }if $\kappa\smin T$
	is fat.
\end{definition}

\begin{lemma}[{Abraham and Shelah \cite[Lemma 1.2]{abraham-shelah-clubshooting}}]
	\label{lem:closed-subset}Assume $\mu<\kappa$, $\kappa$ regular,
	and $S\con\kappa$ has the property that for every club $C\con\kappa$,
	$S\cap C$ contains a closed set of ordinals of order-type $\mu+1$.
	Then for any $\tau<\mu^{+}$ and every club $C\con\kappa$, $S\cap C$
	contains a closed set of order-type $\tau+1$.
\end{lemma}

\begin{definition}
	\label{def:des(S)}Let $S$ be a co-fat subset of some regular $\kappa$.
	We denote by $\des S$ (the \emph{club-shooting for }\textbf{Des}\emph{troying
		the stationarity of $S$})\emph{ }the poset that adds a closed unbounded
	subset of $\kappa\smin S$ using bounded conditions. That is, conditions
	in $\des S$ are closed bounded subsets of $\kappa\smin S$, ordered
	by end-extension -- $p$ is stronger than $q$ ($p\leq q$) iff $p\cap\left(\sup q+1\right)=q$.
\end{definition}

\begin{fact}[{Abraham and Shelah\ \cite[Theorem 1]{abraham-shelah-clubshooting}}]
	\label{fact:club-shooting} Let $\kappa$ be either a strongly inaccessible
	cardinal or the successor of a regular cardinal $\mu$ such that $\mu=\mu^{<\mu}$.
	Let $S\con\kappa$ be co-fat.
	\begin{enumerate}
		\item Forcing with $\des S$ adds a club $C\con\kappa\smin S$.
		\item $\des S$ is $<\kappa$-distributive i.e. forcing with $\des S$ does
		      not add new sets of size $<\kappa$ (hence cardinals and cofinalities
		      $\leq\kappa$ remain unchanged in an extension by $\des S$).
		\item Cardinality of $\des S$ is $2^{<\kappa}$, so if $2^{<\kappa}=\kappa$
		      cardinals above $\kappa$ are not collapsed.
	\end{enumerate}
\end{fact}

One way of obtaining co-fat stationary sets is by using $\square$-sequences.
Recall that for a regular cardinal $\kappa$, the principle $\square_{\kappa}$
(see Jensen's \cite{jensen1972fine}) asserts the existence of a sequence
$\mathcal{C}=\left\langle C_{\alpha}\mid\alpha\in\mathrm{Lim}(\kappa^{+})\right\rangle $
such that
\begin{enumerate}
	\item $C_{\alpha}$ is club in $\alpha$;
	\item If $\beta\in\mathrm{Lim}(C_{\alpha})$ then $C_{\beta}=C_{\alpha}\cap\beta$;
	\item If $\cof{\alpha}<\kappa$ then $\left|C_{\alpha}\right|<\kappa$.
\end{enumerate}
\begin{lemma}
	\label{lem:square-implis-cofat}Let $\kappa$ be regular such that
	$\square_{\kappa}$ holds, and let $S\con E_{<\kappa}^{\kappa^{+}}$
	be stationary. Then there is some stationary $\bar{S}\con S$ which
	is co-fat.
\end{lemma}
\begin{proof}
	Let $\mathcal{C}=\left\langle C_{\alpha}\mid\alpha\in\mathrm{Lim}(\kappa^{+})\right\rangle $
	be a $\square_{\kappa}$ sequence. Note first that if $\alpha<\kappa^{+}$
	is of cofinality $\kappa$, then the order-type of $C_{\alpha}$
	is $\kappa$ -- otherwise, consider the $\kappa+\omega$th element
	of $C_{\alpha}$, say $\beta$. By condition (2), $C_{\beta}=C_{\alpha}\cap\beta$
	so it is of size $\kappa$, which contradicts condition (3) since
	$\cof{\beta}=\omega$.

	Now for every $\gamma\in S$ let $f(\gamma)=\otp(C_{\gamma})$.
	Since $\gamma\in E_{<\kappa}^{\kappa^{+}}$, $f(\gamma)<\kappa$,
	so for every $\gamma\in S\smin\kappa$ $f(\gamma)<\kappa\leq\gamma$,
	hence by Fodor's lemma there is some stationary $\bar{S}\con S$ such
	that $f$ is constant on $\bar{S}$, say with value $\nu$.

	Denote $T=\kappa^{+}\smin\bar{S}$ and we want to show that $T$ is fat.
	By Lemma \ref{lem:closed-subset} it is enough to show that for any
	club $C\con\kappa^{+}$ there is a closed subset of $T\cap C$ of
	order-type $\kappa+1$.
	Let $C$ Be club in $\kappa^{+}$. $T$ contains all ordinals $<\kappa^{+}$ of cofinality $\kappa$ (since $\bar{S}\con E_{<\kappa}^{\kappa^{+}}$)
	so let $\lambda\in T\cap\mathrm{Lim}C$ of cofinality $\kappa$.
	Note that $C_{\lambda} \cap C$ is club in $\lambda$, and by our initial remark it has order-type $\kappa$.
	Let $\gamma$ be a limit point of $C_{\lambda}\cap C$ such that $\otp(C_{\lambda} \cap \gamma)>\nu$.
	Since $C_{\gamma}=C_{\lambda}\cap\gamma$, $f(\gamma) = \otp(C_{\gamma})>\nu$, hence $\gamma\notin\bar{S}$.
	So
	\[
		\{ \gamma \in \mathrm{Lim}(C_{\lambda }\cap C) \mid \otp (C_\gamma)>\nu \} \con T \cap C.
	\]
	Since $\otp(C_{\lambda} \cap C)=\kappa$, this set has order-type $\kappa$.
	It converges to $\lambda \in T \cap C$, hence we
	get a closed subset of $T\cap C$ of order-type $\kappa+1$.
\end{proof}

In order to code sets into $C(\aaa)$, we would like to destroy the
stationarity of specific sets. This in itself can be done by a single
club shooting, but to get models with $V=C(\aaa)$ we'd need to iterate
this construction and to code this coding into $C(\aaa)$, and then
repeat until we catch our tail. For this we need to be able to iterate
club shooting forcing, in a way that later stages don't destroy our
previous codings. This is where the notion of mutual stationarity,
introduced by Foreman and Magidor in \cite{foreman-magidor-mutually-stationary},
steps in.

\subsection{\label{subsec:Shooting-countably-many}Shooting countably many clubs}
\begin{definition}[Foreman and Magidor {\cite[Definition 6]{foreman-magidor-mutually-stationary}}]
	Let $K$ be a collection of regular cardinals with supremum $\delta$
	and suppose that we have $S_{\kappa}\con\kappa$ for each $\kappa\in K$.
	Then the collection $\left\{ S_{\kappa}\mid\kappa\in K\right\} $
	is \emph{mutually stationary }if and only if for all algebras $\mathfrak{A}$
	on $\delta$ (or more generally on $H(\theta)$ for some large enough
	$\theta$) there is an $N\prec\mathfrak{A}$ such that for all $\kappa\in N\cap K$,
	$\sup(N\cap\kappa)\in S_{\kappa}$. We call such $N$ a \emph{witness
		with respect to $\mathfrak{A}$ to the mutual stationarity} of the
	collection.
\end{definition}

The following theorem states that every set of stationary sets of cofinality $\omega$ is mutually stationary,
witnessed by a countable structure:
\begin{theorem}[Foreman and Magidor {\cite[Theorem 7]{foreman-magidor-mutually-stationary}}]
	\label{thm:mutual-stat}Let $\left\langle \kappa_{\alpha}\mid\alpha<\theta\right\rangle $
	be an increasing sequence of regular cardinals with supremum $\delta$,
	and $\left\langle S_{\alpha}\mid\alpha<\theta\right\rangle $ a sequence
	of stationary sets such that $S_{\alpha}\con\kappa_{\alpha}$ consists
	of points of cofinality $\omega$. Then for every algebra $\mathfrak{A}$
	on $\delta$ there is a countable $N\prec\mathfrak{A}$ such that
	for all $\alpha\in N\cap\theta$, $\sup(N\cap\kappa_{\alpha})\in S_{\alpha}$.
\end{theorem}

For our purposes we'll need to strengthen the above theorem a bit,
to incorporate also a stationary subset of $\cpow{\lambda}$ for $\lambda<\kappa_{0}$:
\begin{theorem}
	\label{thm:mutual-stat+}Let $\langle\kappa_{\alpha}\mid\alpha<\theta\rangle$
	be an increasing sequence of regular cardinals with supremum $\delta$.
	Let $\lambda$ be a regular cardinal such that $\lambda^{\omega}<\kappa_{0}$
	. Let $T\con P_{\omega_{1}}(\lambda)$ be stationary in $P_{\omega_{1}}(\lambda)$
	and for $\alpha<\theta$ let $S_{\alpha}\subseteq\kappa_{\alpha}$
	stationary of points of cofinality $\omega$. Then for every $\rho\geq\delta$
	and an algebra $\mathfrak{A}$ on $\rho$ , there is a countable $N\prec\mathfrak{A}$
	such that $N\cap\lambda\in T$ and for $\alpha<\theta$ such that
	$\kappa_{\alpha}\in N$ we have $\sup(N\cap\kappa_{\alpha})\in S_{\alpha}$.
\end{theorem}

\begin{proof}
	We extend the original proof of Theorem \ref{thm:mutual-stat} to incorporate the additional requirement.
	Fix an algebra $\mathfrak{A}$ on $\delta$.
	We wish to find a tree $\mathfrak{T}\con\delta^{<\omega}$ with a \emph{labeling} -- a function
	$l:\mathfrak{T}\to\left\{ \kappa_{\alpha}\mid\alpha<\theta\right\} \cup\left\{ \lambda\right\} $,
	such that the following hold ($\sk^{\mathfrak{A}}$ denotes the Skolem hull in $\mathfrak{A}$):
	\begin{enumerate}
		\item[(1)] If $\sigma\in\mathfrak{T}$ and $l(\sigma)=\kappa_{\alpha}$ then
		      $\left\{ \gamma\mid\sigma^{\frown}\gamma\in\mathfrak{T}\right\} \con\kappa_{\alpha}$
		      and has cardinality $\kappa_{\alpha}$.
		\item[(2)]  If $\sigma\in\mathfrak{T}$ and $\kappa_{\alpha}\in\sk^{\mathfrak{A}}(\sigma)$
		      then there are infinitely many $n\in\omega$ such that if $\tau\supset\sigma$,
		      $\tau\in\mathfrak{T}$ has length $n$, then $l(\tau)=\kappa_{\alpha}$.
		\item[(3)]  If $\sigma\in\mathfrak{T}$ and $l(\sigma)=\lambda$ then
		      there is a unique $\gamma$ such that $\sigma^{\frown}\gamma\in\T$,
		      and this $\gamma$ is $<\lambda$.
		\item[(4)] For every branch $b$ of $\mathfrak{T}$, $\sk^{\mathfrak{A}}(b)\cap\lambda\in T$.
	\end{enumerate}
	This will suffice as the proof of Theorem \ref{thm:mutual-stat} actually shows the following:
	\begin{lemma}
		If $\mathfrak{T}\con\delta^{<\omega}$ and $l:\mathfrak{T}\to\left\{ \kappa_{\alpha}\mid\alpha<\theta\right\} $
		satisfy (1)+(2) above, then there is a decreasing sequence of subtrees
		$\mathfrak{T}_{n}$, the length of the stem of $\mathfrak{T}_{n}$
		is at least $n$, and for the branch $b=\bigcap_{n<\omega}\mathfrak{T}_{n}$,
		$N=\sk^{\mathfrak{A}}(b)$ satisfies that for every $\kappa_{\alpha}\in N$,
		$N\cap\kappa_{\alpha}\in S_{\alpha}$.
	\end{lemma}
	So if $\mathfrak{T}$ also satisfies (4) then we get that $N\cap\lambda=\sk^{\mathfrak{A}}(b)\cap\lambda\in T$,
	and incorporating (3) into the proof of the lemma is straightforward.
	\begin{claim}
		There are $\T,l$ satisfying (2) as above and also:

		(1)\textup{*} If $\sigma\in\mathfrak{T}$ then $\left\{ \gamma\mid\sigma^{\frown}\gamma\in\mathfrak{T}\right\} =l(\sigma)$.
	\end{claim}

	\begin{proof}
		We build $\T,l$ by induction on the length of $\sigma$. Let $\left\langle p_{n}\mid n<\omega\right\rangle $
		be an increasing enumeration of all prime numbers. For each $\sigma\in\lambda^{<\omega}$,
		the set $\{ \alpha<\theta\mid\kappa_{\alpha}\in\sk^{\mathfrak{A}}(\sigma)\} $
		is at most countable. Enumerate it as $\left\langle \alpha_{n}^{\sigma}\mid n<\omega\right\rangle $.

		Start with $l(\left\langle \right\rangle )=\lambda$. Assume that
		$\sigma\in\T$ and $l(\sigma)$ are defined. Then we let the successors
		of $\sigma$ in $\T$ be $\left\{ \sigma^{\frown}\gamma\mid\gamma<l(\sigma)\right\} $
		to get condition (1){*}.
		If the length of $\sigma$ is of the form
		$p_{k}\cdot p_{k+n+1}^{m+1}-1$ for some $k,n,m<\omega$, then we
		let $l(\sigma^{\frown}\gamma)=\kappa_{\alpha_{n}^{\sigma\mets k}}$
		for every $\gamma$. Otherwise $l(\sigma^{\frown}\gamma)$ is arbitrary.
		In this way, if $\sigma\in\T$ and $\kappa_{\alpha}\in\sk^{\mathfrak{A}}(\sigma)$,
		let $k=\mathrm{length}(\sigma)$ and $n$ such that $\kappa_{\alpha}=\kappa_{\alpha_{n}^{\sigma}}$,
		then for every $\tau\supset\sigma$ of length of the form $p_{k}\cdot p_{k+n+1}^{m+1}$
		we have that $l(\tau)=\kappa_{\alpha_{n}^{\tau\mets k}}=\kappa_{\alpha_{n}^{\sigma}}=\kappa_{\alpha}$,
		so condition (2) holds.
	\end{proof}

	Assume $\T,l$ are as in the claim, satisfying (1){*}+(2). We want
	to define a tree $\T'$ satisfying (1)-(4). We begin by defining games
	$\G_{a}$ for $a\in\cpow{\lambda}$ as follows. We have two players,
	B and G, such that the choices of G define a branch through $\T$.
	Fix in advance an enumeration $\left\langle \gamma_{n}\mid n<\omega\right\rangle $
	of $a$. At each play, if $\sigma\in\T$ is defined by the choices
	of G, then:
	\begin{itemize}
		\item If $l(\sigma)\ne\lambda$, B chooses some $D\con l(\sigma)$ of cardinality
		      $<l(\sigma)$, and G chooses some $\gamma\in l(\sigma)\smin D$ such
		      that $\sigma^{\frown}\gamma\in\T$.
		\item If $l(\sigma)=\lambda$ then G chooses the first $\gamma_{n}\in a$
		      not chosen yet.
	\end{itemize}
	G wins if the game defines a branch $b$ through $\T$, such that
	$\sk^{\mathfrak{A}}(b)\cap\lambda=a$. Otherwise B wins. Note that
	this is an open game for B: Since by the construction there are infinitely
	many $n$s such that any node of length $n$ is labeled by $\lambda$,
	we must have $a\con\sk^{\mathfrak{A}}(b)$. So B wins if at some stage
	the game constructs a $\sigma\in\T$ such that $(\sk^{\mathfrak{A}}(\sigma)\cap\lambda)\smin a\ne\emp$,
	hence B's payoff is open. Hence the game is determined.
	\begin{claim}
		There is a club $C\con\cpow{\lambda}$ such that for every $a\in C$,
		G has a winning strategy in $\G_{a}$.
	\end{claim}

	\begin{proof}
		Otherwise, there is a stationary $S\con\cpow{\lambda}$ such that
		for every $a\in S$ B has a winning strategy $s_{a}$ in $\G_{a}$.
		So, for $\theta$ large enough, let
		\[
			N\prec\left\langle H(\theta),\in,\left\langle s_{a}\mid a\in S\right\rangle ,\mathfrak{A},\S,l,\lambda,...\right\rangle
		\]
		be countable such that $N\cap\lambda\in S$ (applying the stationarity
		of the set
		\[
			S\pup H(\theta)=\left\{ X\in\cpow{H(\theta)}\mid X\cap\lambda\in S\right\} ).
		\]
		Denote $N_{0}=N\cap\lambda$. The assumption is that in every play
		of $\G_{N_{0}}$ where B plays according to $s_{N_{0}}$, if the play
		gives the branch $b$ then $\sk^{\mathfrak{A}}(b)\cap\lambda\ne N_{0}$.
		We wish to arrive at a contradiction by describing a play where this
		fails. We  construct this play inductively such that every choice
		G makes is of an ordinal from $N$. So we assume $\sigma\in\T\cap N$
		was defined, and consider two cases:
		\begin{enumerate}
			\item $l(\sigma)\ne\lambda$. In this case B plays according to the strategy
			      $s_{N_{0}}$ a subset $s_{N_{0}}(\sigma)\con l(\sigma)$ of cardinality
			      $<l(\sigma)$. For every $a\in S$ $s_{a}(\sigma)$ is a subset of
			      $l(\sigma)$ of cardinality $<l(\sigma)$. $\left|S\right|=\lambda^{\omega}$
			      which is less than the regular $l(\sigma)$, so the cardinality of
			      $U=\bigcup\left\{ s_{a}(\sigma)\mid a\in S\right\} $ is less than
			      $l(\sigma)$. Since $\sigma,\left\langle s_{a}\mid a\in S\right\rangle \in N$,
			      also $l(\sigma),U\in N$, and so $N\vDash\left|U\right|<l(\sigma)$.
			      So there is some $\gamma\in N\cap l(\sigma)\smin U$ (in particular
			      $\gamma\notin s_{N_{0}}(\sigma)$) such that $\sigma^{\frown}\gamma\in\T$.
			      Then G chooses such $\gamma$.
			\item $l(\sigma)=\lambda$. In this case by the rules of the game G chooses
			      the first element of $N_{0}$ which wasn't chosen yet (according to
			      some enumeration of order-type $\omega$).
		\end{enumerate}
		Let $b$ be the branch constructed in the game. We wish to show that
		$\sk^{\mathfrak{A}}(b)\cap\lambda=N_{0}$. As we noted earlier, by
		the cases $l(\sigma)=\lambda$ we get $\sk^{\mathfrak{A}}(b)\cap\lambda\supseteq N_{0}$.
		On the other hand, since $N$ is in particular closed under the functions
		of $\mathfrak{A}$, and all elements of $b$ were chosen from $N$,
		we get $\sk^{\mathfrak{A}}(b)\cap\lambda\con N\cap\lambda=N_{0}$.
		So indeed $\sk^{\mathfrak{A}}(b)\cap\lambda=N_{0}$, so G wins, contradicting
		the fact that this is a game where B plays according to $s_{N_{0}}$.
	\end{proof}

	Now fix a club $C$ as in the claim,
	by stationarity of $T$ let $a\in T\cap C$, and let $s$ be
	a winning strategy for G in $\G_{a}$. We wish to define a tree $\T'\con\T$
	such that (1)-(3) hold, and every branch in $\T'$ is obtained as
	a play in $\G_{a}$ where G plays according to $s$. This will give
	us (4). Note that (2) will remain true in any subtree of $\T$ such
	that any node is contained in an infinite branch, so the main task
	is to choose nodes in a way preserving (1).
	We work by induction --
	assume that $\sigma\in\T'$ was constructed as a partial play in $\G_{a}$
	where G plays according to $s$, and we determine which of its successors
	to put in next.
	If $l(\sigma)=\lambda$ then, to get (3), there is exactly one
	option -- $\sigma^{\frown}s(\sigma)$ (note (1) doesn't talk about
	the case $l(\sigma)=\lambda$). Otherwise, we inductively define a
	sequence $\left\langle \gamma_{\nu}\mid\nu<l(\sigma)\right\rangle \con l(\sigma)$
	as follows. If $\left\langle \gamma_{\nu}\mid\nu<\mu\right\rangle $
	is already defined, let $\gamma_{\mu}$ be the response of G to B
	playing $\left\{ \gamma_{\nu}\mid\nu<\mu\right\} $. To conclude,
	we let the successors of $\sigma\in\T'$ be $\left\{ \sigma^{\frown}\gamma_{\nu}\mid\nu<l(\sigma)\right\} $.
	So condition (1) is satisfied, and by the construction any
	branch through $\T'$ is given as a play according to $s$, as required.
\end{proof}

The following is a useful remark made at the beginning of \cite[section 7]{foreman-magidor-mutually-stationary}:
\begin{remark}\label{rem:mutual-stat}
	\begin{enumerate}
		\item Any subset of a sequence of mutually stationary sets is mutually stationary.
		\item If $N\prec\left\langle H(\theta),\in,\Delta,...\right\rangle =\mathfrak{A}$
		      (where $\theta>\delta$ is regular) and $\nu\in N$, then for all
		      regular $\mu\in N\smin\left(\nu+1\right)$,
		      \[
			      \sup(N\cap\mu)=\sup(\sk^{\mathfrak{A}}(N\cup\nu)\cap\mu).
		      \]
		      In particular, if $N$ witnesses the mutual stationarity of a sequence
		      $\left\langle S_{\kappa}\mid\kappa\in K\right\rangle $ with respect
		      to $\mathfrak{A}$, $K\con N$ and $\nu\in N\cap\min K$ then also
		      $\sk^{\mathfrak{A}}(N\cup\nu)$ is such a witness.
	\end{enumerate}
\end{remark}

We now show that using mutual stationarity allows us to iterate countably
many club shooting forcings, in a way which is $\omega$-distributive
and preserves the stationarity of sets we do not wish to destroy.
\begin{theorem}
	\label{Thm:it-club-shooting-countable}Let $K=\left\langle \kappa_{\alpha}\mid\alpha<\theta\right\rangle $
	be an increasing sequence of uncountable cardinals such that for every
	regular $\lambda$ and every $\kappa\in K$, $\lambda<\kappa\Rightarrow\lambda^{\omega}<\kappa$.
	\footnote{This will hold for example if there is $\mu<\kappa_{0}$ such that
		GCH holds starting at $\mu$.}
	Let $\delta=\sup K$. For each $\alpha<\theta$ fix some stationary $T_{\alpha}\con E_{\omega}^{\kappa_{\alpha}}$.
	Assume that $\langle P_{\alpha},\dot{Q}_{\alpha}\mid\alpha<\theta\rangle$
	is a forcing iteration with \textbf{countable} support, and for each
	$\alpha$ there is a $P_{\alpha}$ name $\dot{S}_{\alpha}$ such that
	\begin{quote}
		$P_{\alpha}\Vdash$``\,$\dot{S}_{\alpha}$ is a co-fat subset of $\kappa_{\alpha}$,
		$\check{T}_{\alpha}\cap\dot{S}_{\alpha}=\emp$, and $\dot{Q}_{\alpha}=\des{\dot{S}_{\alpha}}$''.
	\end{quote}
	Let $\lambda$ be regular.
	\begin{enumerate}
		\item \label{enu:distributivity-omega}$P_{\theta}$ is $\omega$-distributive.
		\item \label{enu:ord-stationary-count}Let $T\con E_{\omega}^{\lambda}$
		      stationary (in $V$) such that if $\lambda=\kappa_{\alpha}$ for some
		      $\alpha$, then $P_{\theta}\Vdash\check{T}\cap\dot{S}_{\alpha}=\emp$.
		      Then $P_{\theta}\Vdash$``\,$\check{T}$ is stationary''.
		\item \label{enu:general-stat-count}Let $T\con\power_{\omega_{1}}(\lambda)$
		      be stationary (in $V$) such that  for every $a\in T$ and every
		      $\kappa_{\alpha}\in a$ (and if $\lambda=\kappa_{\alpha}$ then also
		      for $\lambda$), $P_{\theta}\Vdash\sup(a\cap\kappa_{\alpha})\notin\dot{S}_{\alpha}$.
		      Then $P_{\theta}\Vdash$``\,$T$ is stationary''.
		\item \label{enu:calculate-stat-count} Assume that for every $\alpha$
		      $\dot{Q}_{\alpha}=\des{\check{S}_{\alpha}}$ for $S_{\alpha}\in V$.
		      \begin{enumerate}
			      \item Let $T\con\power_{\omega_{1}}(\lambda)$, $T\in V$.
			            Then $P_{\theta}\Vdash$``\,$T$ is stationary''
			            iff $V\vDash$``\,$\tilde{T}$ is stationary'' where
			            \[
				            \tilde{T}=\left\{ a\in T\mid\forall\kappa_{\alpha}\in a\,\sup(a\cap\kappa_{\alpha})\notin S_{\alpha}\right\} .
			            \]
			      \item $C(\aaa)^{V^{P_{\theta}}}\con V$.
		      \end{enumerate}
	\end{enumerate}
\end{theorem}

\begin{proof}
	For (\ref*{enu:distributivity-omega}), let $\left\{ D_{n}\mid n<\omega\right\} \con P_{\theta}$
	be open dense, $p\in P_{\theta}$ and we wish to find $p'\leq p$
	in $\bigcap_{n<\omega}D_{n}$. Consider the set $\left\{ T_{\alpha}\mid\alpha<\theta\right\} $
	of stationary sets, each consisting of points of cofinality $\omega$.
	By Theorem \ref{thm:mutual-stat}, there is some countable
	\[
		N\prec\left\langle \delta,K,P_{\theta},p,\left\{ T_{\alpha}\mid\alpha<\theta\right\} ,D_{k},\alpha\right\rangle _{k<\omega,\alpha\in\supp(p)}
	\]
	(note that $\supp(p)$ is countable so this is possible) such that
	for every $\kappa_{\alpha}\in N$, $\sup N\cap\kappa_{\alpha}\in T_{\alpha}$.
	Enumerate all the dense open sets of $P_{\theta}$ in $N$ by $\left\langle d_{n}\mid n<\omega\right\rangle $,
	and inductively define $p_{n}\in N\cap P_{\theta}$, by setting $p_{0}=p$
	and choosing $p_{n+1}\in N$ to be some extension of $p_{n}$ in $d_{n}$.
	Let $s=\bigcup\supp(p_{n})$ (note that $s\con N$ and is countable),
	and define $p'$ to be the function on $\theta$ such that for $\alpha\in s$,
	$p'(\alpha)=\bigcup_{n<\omega}p_{n}(\alpha)\cup\left\{ \sup\bigcup_{n<\omega}p_{n}(\alpha)\right\} $
	and for $\alpha\in\theta\smin s$ $p'(\alpha)=\emp$. By a density
	argument, for every $\alpha\in s$, $\bigcup_{n<\omega}p_{n}(\alpha)$
	is cofinal in $\sup N\cap\kappa_{\alpha}$, and by assumption $\sup N\cap\kappa_{\alpha}\in T_{\alpha}$
	which is forced to be disjoint from $\dot{S}_{\alpha}$, so $p'$
	is indeed a condition in $P_{\theta}$ with support $s$ extending
	each $p_{n}$. Since each $d_{n}$ is open, $p'\in d_{n}$, and since
	every $D_{k}$ is also open dense in $N$, it equals some $d_{n}$,
	so $p'\in\bigcap_{k<\omega}D_{k}$ as required.

	(\ref*{enu:ord-stationary-count}) follows from (\ref*{enu:general-stat-count})
	by using Lemma \ref{lem:correspond}.\ref*{enu:corres-stat}, so we
	only prove (\ref*{enu:general-stat-count}). Let $p\in P_{\theta}$
	and $\dot{F}$ be a $P_{\theta}$ name forced by $p$ to be a function
	from $[\lambda]^{<\omega}$ to $\lambda$. We wish to find $p'\leq p$
	forcing that there is an element of $T$ closed under $\dot{F}$.
	Let $\mathfrak{A}=\left\langle H(\mu),f_{n}\right\rangle _{n<\omega}$
	be an algebra (for $\mu$ large enough) extending
	\[
		\left\langle H(\mu),K,P_{\theta},p,\lambda,\dot{F},\left\{ T_{\alpha}\mid\alpha<\theta\right\} ,T,\alpha\right\rangle _{\alpha\in\supp(p)}
	\]
	(note that $\supp(p)$ is countable so this is possible) such that
	the functions of $\mathfrak{A}$ are closed under composition.

	By Theorem \ref{thm:mutual-stat+} applied to $T$ and $\left\{ T_{\alpha}\mid\alpha<\theta\land\kappa_{\alpha}>\lambda\right\} $,
	there is $N\prec\mathfrak{A}$ countable such that $N\cap\lambda\in T$
	and for every $\kappa_{\alpha}\in N\smin(\lambda+1)$, $\sup N\cap\kappa_{\alpha}\in T_{\alpha}$
	(if $\lambda\geq\sup K$ then simply use the stationarity of $T$).
	Now as in the previous clause, we can define a decreasing sequence
	of conditions $\left\langle p_{n}\mid n<\omega\right\rangle \con N$,
	starting with $p_{0}=p$, that meets every dense open set of $N$.
	As before let $s=\bigcup\supp(p_{n})\con N$, and define $p'$ to
	be the function on $\theta$ such that for $\alpha\in s$, $p'(\alpha)=\bigcup_{n<\omega}p_{n}(\alpha)\cup\left\{ \sup\bigcup_{n<\omega}p_{n}(\alpha)\right\} $
	and for $\alpha\in\theta\smin s$ $p'(\alpha)=\emp$. $p'$ will be
	a condition if we show that for every $\alpha\in s$, $\Vdash\sup\bigcup_{n<\omega}p_{n}(\alpha)\notin\dot{S}_{\alpha}$.
	Again by a density argument, for every $\alpha\in s$ (in particular $\alpha\in N$), $\sup\bigcup_{n<\omega}p_{n}(\alpha)=\sup N\cap\kappa_{\alpha}$.
	If $\kappa_{\alpha}>\lambda$ then this is in $T_{\alpha}$ which
	is forced to be disjoint from $\dot{S}_{\alpha}$. If $\kappa_{\alpha}\le\lambda$,
	then since $N\cap\lambda\in T$, the assumption is exactly that $P_{\theta}\Vdash\sup(N\cap\kappa_{\alpha})\notin\dot{S}_{\alpha}$.
	Hence $p'\in P_{\theta}$.

	Now, $p_{0}$ forces that $\dot{F}$ is a function on $[\lambda]^{<\omega}$,
	so for each $x\in N\cap[\lambda]^{<\omega}$ there is in $N$ a dense
	set of conditions determining the value of $\dot{F}(\check{x})$, which is an element
	of $N\cap\lambda$, and so $p'$ forces some value for $\dot{F}(\check{x})$. So
	eventually $p'$ determines some function $\tilde{F}$ with domain
	$[\lambda]^{<\omega}\cap N$ and range in $N$ such that $p'\Vdash\dot{F}\mets N\cap\lambda=\tilde{F}$.
	So $p'$ forces that $N\cap\lambda\in T$ is closed under $\dot{F}$,
	as required.

	For (\ref*{enu:calculate-stat-count})(a), first note that by (\ref*{enu:general-stat-count})
	if $V\vDash$``$\tilde{T}$ is stationary'' then $P_{\theta}\Vdash$``$\tilde{T}$
	is stationary'' but $\tilde{T}\con T$ so also $P_{\theta}\Vdash$``$T$
	is stationary''.
	Second, assume $P_{\theta}\Vdash$``$T$ is stationary''
	but $V\vDash$``$\tilde{T}$ is not stationary''.
	Let $G\con P_{\theta}$ be generic, and work in $V[G]$.
	In $V[G]$, $T$ is stationary while $\tilde{T}$ is not, hence
	\[
		T\smin\tilde{T}=\left\{ a\in T\mid\exists\kappa_{\alpha}\in a(\sup(a\cap\kappa_{\alpha})\in S_{\alpha})\right\}
	\]
	is stationary in $\cpow{\lambda}^{V[G]}$.
	Note that in this case we must have $\lambda>\kappa_0$.
	Let $F:\lambda \times  \lambda \to \lambda$ be a function such  that for every  $\alpha$ such that $\kappa_\alpha<\lambda$ and $\beta<\lambda$,
	$F(\kappa_\alpha,\beta) = \min (G \cap \kappa_\alpha \smin \beta)$.
	Then	by stationarity there is $b \in T \smin \tilde{T}$ closed under $F$.
	So on one hand, there is $\kappa_{\alpha}\in b$ such that
	$\sup(b\cap\kappa_{\alpha})\in S_{\alpha}$.
	But on the other hand, $b$ is closed under $F(\kappa_{\alpha},\cdot)$ so $b \cap \kappa_\alpha$ contains unboundedly many elements of $G$, which implies  (as $G$ is closed) that $\sup(b\cap\kappa_{\alpha})\in G$.
	But $G$ is disjoint from $S_{\alpha}$, a contradiction.

	For (\ref*{enu:calculate-stat-count})(b), let $W$ be a generic extension
	of $V$ by $P_{\theta}$ and denote the $\alpha$th level in the construction
	of $C(\aaa)^{W}$ by $L'_{\alpha}$. We prove by induction that for
	every $\alpha$ the construction up to $\alpha$ can be done in $V$
	-- so the sequence $\left\langle L_{\alpha}'\mid\alpha\in\ord\right\rangle $
	is definable in $V$ and in particular every level is both contained
	in $V$ and an element of $V$. $L_{0}'$ is clear and the limit case
	is immediate from the induction hypothesis. So assume the assumption
	holds up to (and including) $\alpha$ and $X\in\Def_{\LL(\aaa)}(L_{\alpha}')$,
	i.e. for some $\fii\in\LL(\aaa)$, $\boldsymbol{b}\in\left.L_{\alpha}'\right.^{<\omega}$
	and $\boldsymbol{t}\in\cpow{L_{\alpha}'}^{<\omega}$
	\[
		X=\left\{ a\in L_{\alpha}'\mid W\vDash\left[\left(L_{\alpha}',\in\right)\vDash_{\LL(\aaa)}\fii\left(a,\boldsymbol{b},\boldsymbol{t}\right)\right]\right\} .
	\]
	We prove by induction on the complexity of $\fii$ that the relation
	\[
		W\vDash\left[\left(L_{\alpha}',\in\right)\vDash_{\LL(\aaa)}\fii\left(a,\boldsymbol{b},\boldsymbol{t}\right)\right]
	\]
	can in fact be determined in $V$. The only interesting case is when
	$\fii=\stat s\psi$ for some $\psi\in\LL(\aaa)$, so set $\mathcal{M}=\left(L_{\alpha}',\in\right)$
	and we need to show that
	\[
		W\vDash\left(\mathcal{M}\vDash\stat s\psi(s,a,\boldsymbol{b},\boldsymbol{t})\right)
	\]
	can by determined in $V$. Note that by the induction assumption
	on $L_{\alpha}'$ and $\omega$-dis\-tri\-buti\-vity, $a,\boldsymbol{b},\boldsymbol{t}\in V$,
	and also the countable sets $s$ considered in evaluating $\psi$
	are the same in $V$ and $W$. So we assume inductively that for any
	such $s,a,\boldsymbol{b},\boldsymbol{t}$, the relation $W\vDash\left(\mathcal{M}\vDash\psi(s,a,\boldsymbol{b},\boldsymbol{t})\right)$
	can be determined in $V$. Now for fixed $a,\boldsymbol{b},\boldsymbol{t}$,
	the relation $W\vDash\left(\mathcal{M}\vDash\stat s\psi(s,a,\boldsymbol{b},\boldsymbol{t})\right)$
	is equivalent to ``$W\vDash T$ is stationary'' where
	\[
		T=\left\{ s\in\cpow{L_{\alpha}'}\mid W\vDash\left(\mathcal{M}\vDash\psi(s,a,\boldsymbol{b},\boldsymbol{t})\right)\right\} .
	\]
	$T\in V$ since $W\vDash\left(\mathcal{M}\vDash\psi\left(a,\boldsymbol{b},\boldsymbol{t}\right)\right)$
	can be determined in $V$, so by clause (\ref*{enu:calculate-stat-count})(a),
	``$W\vDash T$ is stationary'' is equivalent to ``$V\vDash\tilde{T}$
	is stationary'' where
	\[
		\tilde{T}=\left\{ a\in T\mid\forall\kappa_{\alpha}\in a(\sup(a\cap\kappa_{\alpha})\notin S_{\alpha})\right\} \in V,
	\]
	so $W\vDash\left(\mathcal{M}\vDash\stat s\psi(s,\boldsymbol{b},\boldsymbol{t})\right)$
	can also be determined in $V$.

	Wrapping up, we can correctly determine in $V$ for every $a\in L_{\alpha}'$
	if it \emph{will} satisfy $\fii=\stat s\psi$ after the extension,
	so in fact we can calculate $X$ in $V$. So every element of $L_{\alpha+1}'$
	is in $V$, and the entire $L_{\alpha+1}'$ can be computed in $V$.
\end{proof}
This theorem will allow us to force over $L$, and even over generic extensions
of $L$, with a countable iteration of club shooting forcing. This
will be enough to get a proper extension of the ground model satisfying
$V=C(\aaa)$ (see Theorem \ref{thm:coding}), and also extensions with
a decreasing $C(\aaa)$ sequence of  length $\omega$. However, to
get longer sequences we will need longer iterations while still preserving
enough distributivity. For this it will not be enough to have a mutually
stationary sequence of sets each of which is co-fat on its own --
we will need a notion of ``mutual fatness''.

\subsection{\label{subsec:Shooting-uncountably-many}Shooting uncountably many
	clubs}
\begin{definition}
	\label{def:mutual-fatness}Let $K$ be a collection of regular cardinals
	with supremum $\delta$ and suppose that we have $S_{\kappa}\con\kappa$
	for each $\kappa\in K$. Then the collection $\left\{ S_{\kappa}\mid\kappa\in K\right\} $
	is \emph{mutually $\nu$-fat }(for $\nu<\min K$)\emph{ }if and only
	if for all algebras $\mathfrak{A}$ on $\delta$ (or more generally
	on $H(\theta)$ for some large enough $\theta$) there is a sequence
	of models $\left\langle N_{\alpha}\mid\alpha\leq\nu\right\rangle $
	satisfying that for all $\alpha\leq\nu$:
	\begin{enumerate}
		\item $N_{\alpha}\prec\mathfrak{A}$;
		\item $\left\langle N_{\beta}\mid\beta\leq\alpha\right\rangle \in N_{\alpha+1}$
		      (the sequence is \emph{internally approachable});
		\item For limit $\alpha$ $N_{\alpha}=\bigcup_{\beta<\alpha}N_{\beta}$
		      (the sequence is \emph{continuous});
		\item If $\kappa\in N_{\alpha}\cap K$ then $\sup(N_{\alpha}\cap\kappa)\in S_{\kappa}$.
	\end{enumerate}
	We say the collection is \emph{mutually ${<}\mu$-fat }if it is mutually
	$\nu$-fat for every regular $\nu<\mu$ and that it is simply \emph{mutually
		fat }if it is mutually ${<}\min K$-fat. We say that the collection
	is \emph{strongly} mutually ($\nu$-)fat if for every algebra,
	mutual ($\nu$-)fatness is witnessed by models containing $K$ as a subset (or equivalently
	by models containing $\left|K\right|$ as a subset).

\end{definition}

\begin{remark}
	Note that if $\left\{ S_{\kappa}\mid\kappa\in K\right\} $ are mutually
	$\nu$-fat witnessed with respect to some $\mathfrak{A}$ by a sequence
	$\left\langle N_{\alpha}\mid\alpha\leq\nu\right\rangle $ then for
	every regular $\mu\leq\nu$, $N_{\mu}$ witnesses with respect to
	$\mathfrak{A}$ that the sequence $\left\{ S_{\kappa}\cap E_{\mu}^{\kappa}\mid\kappa\in K\right\} $
	is mutually stationary.

\end{remark}

\begin{theorem}
	\label{Thm:it-club-shooting-full}Let $K=\left\langle \kappa_{\alpha}\mid\alpha<\theta\right\rangle $
	be an increasing sequence of uncountable regular cardinals, $\theta<\kappa_{0}$.
	Let $\delta=\sup K$. 
	Assume that $\mathcal{T}=\left\langle T_{\alpha}\mid\alpha<\theta\right\rangle $
	is a strongly mutually fat sequence ($T_{\alpha}\con\kappa_{\alpha}$).
	Assume that $\langle P_{\alpha},\dot{Q}_{\alpha}\mid\alpha<\theta\rangle$
	is a forcing iteration with \textbf{full} support,   and for each
	$\alpha$ there is a $P_{\alpha}$-name $\dot{S}_{\alpha}$ such that
	\begin{quote}
		$P_{\alpha}\Vdash$ ``$\dot{S}_{\alpha}\con E_{\omega}^{\kappa_{\alpha}}\smin \check{T}_{\alpha}$
		and $\dot{Q}_{\alpha}=\des{\dot{S}_{\alpha}}$''.
	\end{quote}

	\begin{enumerate}
		\item  \label{enu:distributivity-full}$P_{\theta}$ is ${<}\kappa_{0}$-distributive.
	\end{enumerate}
	Let $\lambda$ be regular and assume additionally that GCH holds from some cardinal below $\kappa_{0}$.
	\begin{enumerate}
		\setcounter{enumi}{1}
		\item \label{enu:ord-stationary-full} Let $T\con E_{\omega}^{\lambda}$ stationary (in $V$) such that if
		      $\lambda=\kappa_{\alpha}$ for some $\alpha$,
		      then $P_{\theta}\Vdash \check{T}\cap\dot{S}_{\alpha}=\emp$.
		      Then $P_{\theta}\Vdash$``\,$\check{T}$ is stationary''.
		\item \label{enu:general-stat-full}Let $T\con\power_{\omega_{1}}(\lambda)$
		      be stationary (in $V$) such that  for every $a\in T$ and every
		      $\kappa_{\alpha}\in a$, $P_{\theta}\Vdash\sup(a\cap\kappa_{\alpha})\notin\dot{S}_{\alpha}$,
		      and if $\lambda=\kappa_{\alpha}$ then also $P_{\theta}\Vdash\sup(a)\notin\dot{S}_{\alpha}$
		      Then $P_{\theta}\Vdash$ ``\,$\check{T}$ is stationary''.
		\item \label{enu:calculate-stat-full} Assume that there is $\left\langle S_{\alpha}\mid\alpha<\theta\right\rangle \in V$
		      such that for every $\alpha<\theta$ $\dot{Q}_{\alpha}=\des{\check{S}_{\alpha}}$.
		      \begin{enumerate}
			      \item Let $T\con\power_{\omega_{1}}(\lambda)$, $T\in V$. Then $P_{\theta}\Vdash$ ``\,$\check{T}$ is stationary''
			            iff $V\vDash$ ``\,$\tilde{T}$ is stationary'' where
			            \[
				            \tilde{T}=\left\{ a\in T\mid\forall\alpha<\theta\left[\sup a\notin S_{\alpha}\land(\kappa_{\alpha}\in a\to\sup(a\cap\kappa_{\alpha})\notin S_{\alpha})\right]\right\} .
			            \]
			      \item $C(\aaa)^{V^{P_{\theta}}}\con V$.
		      \end{enumerate}
	\end{enumerate}
\end{theorem}

\begin{proof}
	For (\ref*{enu:distributivity-full}) let $\bar{D}=\left\langle D_{\alpha}\mid\alpha<\nu\right\rangle $,
	$\nu<\kappa_{0}$, be a sequence of dense open sets in $P_{\theta}$
	and let $p\in P_{\theta}$. We can assume without loss of generality
	that $\nu$ is a regular cardinal (for singular $\nu$ we can work
	by induction and set $\bigcap_{\alpha<\nu}D_{\alpha}=\ensuremath{\bigcap_{\beta<\cof{\nu}}(\bigcap_{\alpha<\gamma_{\beta}}D_{\alpha})}$
	for some $\left\langle \gamma_{\beta}\mid\beta<\cof{\nu}\right\rangle $
	cofinal in $\nu$). For convenience we will assume the dense open
	sets are indexed only by successor ordinals. From the assumption of
	strong mutual fatness applied to $\mathfrak{A}=\left\langle H(\lambda),\in,\spo,K,P_{\theta},p,\bar{D}...\right\rangle $
	(where $\spo$ is a well-order of $H(\lambda)$) there is an internally
	approachable continuous sequence of elementary submodels of $\mathfrak{A}$,
	$\left\langle N_{\alpha}\mid\alpha\leq\nu\right\rangle $, such that
	$P_{\theta},p,\bar{D}\in N_{0}$, $\max(\theta,\nu)\con N_{0}$, and
	for all $\alpha\leq\nu$ and $\beta<\theta$, $\sup(N_{\alpha}\cap\kappa_{\beta})\in T_{\beta}$.
	We define a decreasing sequence of conditions in $P_{\theta}$ --
	$\left\langle p_{\alpha}\mid\alpha\leq\nu\right\rangle $ --
	such that for every $\alpha<\nu$\ $\left\langle p_{\beta}\mid\beta\leq\alpha\right\rangle \in N_{\alpha+1}$,
	for successor $\alpha$ $p_{\alpha}\in D_{\alpha}$ and for every
	$\beta<\theta$ $\max p_{\alpha}(\beta)=\sup(N_{\alpha}\cap\kappa_{\beta})$.\footnote{Strictly speaking we should say that for every $\beta<\theta$, $p_{\alpha}\mets\beta$
		forces that $\max p_{\alpha}(\beta)=\sup(N_{\alpha}\cap\kappa_{\beta})$,
		but for clarity we would not mention this explicitly.} Begin with $p_{0}=p\in N_{0}\in N_{1}$. If $\left\langle p_{\beta}\mid\beta\leq\alpha\right\rangle \in N_{\alpha+1}$
	is defined, choose the $\spo$-least $p_{\alpha+1}'\in D_{\alpha+1}\cap N_{\alpha+1}$
	extending $p_{\alpha}$, and define $p_{\alpha+1}$ such that for
	every $\beta<\theta$ $p_{\alpha+1}(\beta)=p_{\alpha+1}'(\beta)\cup\{\sup(N_{\alpha+1}\cap\kappa_{\beta})\}$.
	By our choice $\sup p_{\alpha+1}'(\beta)<\sup(N_{\alpha+1}\cap\kappa_{\beta})\in T_{\beta}\cap N_{\alpha+2}$
	so this is indeed a condition in $N_{\alpha+2}$, and so also $\left\langle p_{\beta}\mid\beta\leq\alpha+1\right\rangle \in N_{\alpha+2}$.

	Now assume $\alpha\leq\nu$ is limit and $\left\langle p_{\gamma}\mid\gamma<\alpha\right\rangle $
	are defined such that for every $\alpha'<\alpha$ $\left\langle p_{\gamma}\mid\gamma\leq\alpha'\right\rangle \in N_{\alpha'+1}$.
	The sequence $\left\langle N_{\gamma}\mid\gamma<\alpha\right\rangle $
	is increasing and continuous, so by the induction hypothesis
	\[
		\sup\bigcup_{\gamma<\alpha}p_{\gamma}(\beta)=\sup\left\{ \sup(N_{\gamma}\cap\kappa_{\beta})\mid\gamma<\alpha\right\} =\sup(N_{\alpha}\cap\kappa_{\beta})\in T_{\beta}
	\]
	so we can define $p_{\alpha}\in P_{\theta}$ by letting $p_{\alpha}(\beta)=\bigcup_{\gamma<\alpha}p_{\gamma}(\beta)\cup\{\sup\bigcup_{\gamma<\alpha}p_{\gamma}(\beta)\}$
	for every $\beta<\theta$.
	If $\alpha<\nu$, $\left\{ p_{\gamma}\mid\gamma<\alpha\right\} \con N_{\alpha}\in N_{\alpha+1}$,
	but furthermore, the recursive definition of the sequence can in fact
	be performed inside $N_{\alpha+1}$, and will yield the same result,
	so $\left\langle p_{\gamma}\mid\gamma<\alpha\right\rangle \in N_{\alpha+1}$,
	and hence $p_{\alpha}\in N_{\alpha+1}$ as required. For $\alpha=\nu$
	this is of-course not required.
	Now, since the sets $D_{\alpha}$
	are open, we get inductively that $p_{\alpha}\in\bigcap_{\gamma<\alpha}D_{\gamma}$
	and in particular $p_{\nu}\in\bigcap_{\gamma<\nu}D_{\gamma}$ as required.

	For (\ref{enu:ord-stationary-full}) and (\ref{enu:general-stat-full}),
	denote by $\bar{\lambda}$ the first $\alpha$ such that $\lambda\leq\kappa_\alpha$ if there is one, and $\theta$ otherwise.
	First assume that $\lambda\notin K$.
	Then
	\[
		P_{\theta}=P_{<\bar\lambda}*P_{\bar\lambda<}=\left\{ p\mets\bar\lambda\mid p\in P_{\theta}\right\} *\left\{ p\mets \theta\smin\bar\lambda\mid p\in P_{\theta}\right\}
	\]
	(if $\bar{\lambda}=\theta$ then $P_{<\bar{\lambda}}=P_{\theta}$ and $P_{\bar{\lambda}<}$ is trivial).
	$P_{\bar\lambda<}$ is $\lambda$-distributive (or trivial), so doesn't destroy stationary
	sets at $\lambda$ or at $\cpow{\lambda}$.
	For $P_{<\bar\lambda}$ the argument is essentially the same as in Theorem \ref{Thm:it-club-shooting-countable}(\ref*{enu:general-stat-count})
	(and even simpler):
	For any $p\in P_{<\bar\lambda}$ and $\dot{F}$ a
	$P_{<\bar\lambda}$ name forced by $p$ to be a function from $[\lambda]^{<\omega}$
	to $\lambda$, let $\mathfrak{A}=\left\langle H(\mu),f_{n}\right\rangle _{n<\omega}$
	be an algebra (for $\mu$ large enough) extending
	\[
		\left\langle H(\mu),K,P_{\theta},p,\lambda,\dot{F},\left\{ T_{\alpha}\mid\alpha<\lambda\right\} ,T\right\rangle
	\]
	such that the functions of $\mathfrak{A}$ are closed under composition.
	By stationarity of $T$ there is some countable $N\prec\mathfrak{A}$
	such that $N\cap\lambda\in T$. As before we define a decreasing sequence
	of $P_{<\bar{\lambda}}$ conditions $\left\langle p_{n}\mid n<\omega\right\rangle \con N$,
	starting with $p_{0}=p$, that meets every dense open set of $N$,
	and let $p'=\bigcup_{n<\omega}p_{n}$.
	Since we assumed that for every $a\in T$ and every $\kappa_{\alpha}\in a$, $P_{\theta}\Vdash\sup(a\cap\kappa_{\alpha})\notin\dot{S}_{\alpha}$,
	$p'$ is a condition of $P_{<\bar{\lambda}}$.
	As before, $p'$ forces that $N\cap\lambda\in T$ is closed under $\dot{F}$, as required.

	Now assume $\lambda\in K$, so that $\lambda=\kappa_{\bar\lambda}$.
	Then $P_{\theta}=P_{<\bar\lambda}*P_{\bar\lambda}*P_{\bar\lambda<}$
	where as before $P_{<\bar\lambda}$ and $P_{\bar\lambda<}$ don't destroy
	stationarity at $\lambda$, and for $P_{\bar\lambda}$ we can apply theorem
	\ref{Thm:it-club-shooting-countable} to get that stationarity of
	$T$ is preserved.

	The proof of (\ref*{enu:calculate-stat-full}) is the same as in (\ref*{enu:calculate-stat-count})
	of Theorem \ref{Thm:it-club-shooting-countable}.
\end{proof}
We will now present two methods of obtaining strongly mutually fat
sets -- by using $\square$-sequences, and by forcing a sequence
of non-reflecting stationary sets.

\subsubsection{\label{subsec:square-sequence}$\square$-sequence on singulars}

Recall the so-called ``Global $\square$ principle'' (see Jensen's
\cite{jensen1972fine}) that asserts the existence of a sequence $\mathcal{C}=\left\langle C_{\alpha}\mid\alpha\text{ a singular ordinal}\right\rangle $
such that
\begin{enumerate}
	\item $C_{\alpha}$ is club in $\alpha$;
	\item $\otp(C_{\alpha})<\alpha$;
	\item If $\beta\in\mathrm{Lim}(C_{\alpha})$ then $C_{\beta}=C_{\alpha}\cap\beta$.
\end{enumerate}
For an ordinal $\theta$ we say that $\square_{(\theta)}$ holds if there is a sequence
\[
	\mathcal{C}=\left\langle C_{\alpha}\mid\alpha<\theta\text{ a singular limit ordinal}\right\rangle
\]
satisfying (1-3) above (this corresponds to the principle $\square_{(0,\theta)}$
in \cite[definition 1.5]{DZAMONJA-shelah}). We call such a sequence
a \emph{$\square_{(\theta)}$-sequence.}

Denote the class of singular limit ordinals by $\sing$ and for any
$\theta$ denote $\theta\cap\sing$ by $\sing[\theta]$. Note that
$\sing[\theta]$ is stationary in $\theta$.
\begin{theorem}
	\label{thm:mut-fat-squar}Let $K=\left\{ \kappa_{\xi}\mid\xi<\theta\right\} $
	be a set of regular uncountable cardinals such that $\theta^{+}<\kappa_{0}$,
	$\kappa_{0}$ not the successor of a singular, and that the GCH holds
	starting from some cardinal $\psi$ such that $\psi^{++}<\kappa_{0}$.
	Let $\kappa^{*}=\sup K$ and assume $\mathcal{C}=\left\langle C_{\alpha}\mid\alpha\in\sing[\kappa^{*}+1]\right\rangle $
	is a \emph{$\square_{(\kappa^{*}+1)}$}-sequence and $\left\langle S_{\kappa}\mid\kappa\in K\right\rangle $
	is a sequence of sets such that $S_{\kappa}\con E_{\omega}^{\kappa}$
	is stationary in $\kappa$. Then there is a sequence $\left\langle \bar{S}_{\kappa}\mid\kappa\in K\right\rangle $
	such that for every $\kappa\in K$ $\bar{S}_{\kappa}\con S_{\kappa}$
	is stationary in $\kappa$ and the sequence $\left\langle T_{\kappa}=\kappa\smin\bar{S}_{\kappa}\mid\kappa\in K\right\rangle $
	is strongly mutually fat.
\end{theorem}

\begin{proof}
	We begin as in Lemma \ref{lem:square-implis-cofat}. Considering the
	function $\gamma\mapsto\otp(C_{\gamma})$, we apply Fodor's lemma
	at every $S_{\kappa}$ for $\kappa\in K$ to obtain $\left\langle \bar{S}_{\kappa}\mid\kappa\in K\right\rangle $
	a sequence of stationary sets, $\bar{S}_{\kappa}\con S_{\kappa}$
	and $\vec{\nu}=\left\langle \nu_{\kappa}\mid\kappa\in K\right\rangle \in\prod_{\kappa\in K}\kappa$
	such that for every $\gamma\in\bar{S}_{\kappa}$ $\otp(C_{\gamma})=\nu_{\kappa}$.
	Set $T_{\kappa}=\kappa\smin\bar{S}_{\kappa}$.  Note that if $\alpha$
	is a limit point of some $C_{\gamma}$, then $C_{\alpha}=\alpha\cap C_{\gamma}$
	implies that $\otp(C_{\alpha})$ equals $\alpha$'s place in $C_{\gamma}$,
	so if $\alpha$'s place in $C_{\gamma}$ is not $\nu_{\kappa}$, then
	it is in $T_{\kappa}$.

	Let $\mathfrak{A}$ be an algebra on some large enough $H(\lambda)$.
	We can assume it has constants for $\kappa^{*}$, $K$, $\mathcal{C}$
	and $\vec{\nu}$, and fix a cardinal $\mu<\kappa_{0}$ which we can
	assume is regular (as $\kappa_{0}$ is not the successor of a singular).
	We need to find an internally approachable continuous sequence of models
	$N_{\alpha}\prec\mathfrak{A}$, $\alpha\leq\mu$ such that
	$K\con N_{0}$ and for all $\alpha\leq\mu$ and $\kappa\in K$, $\sup(N_{\alpha}\cap\kappa)\in T_{\kappa}$.

	Let $\bar{\mu}=\max\{\theta^{+},\psi^{++},\mu\}$ (note $\bar{\mu}<\kappa_{0}$)
	and let $M\prec\mathfrak{A}$ of size $\bar{\mu}$ such that $\bar{\mu}+1,K\con M$
	and $\psu{<\bar{\mu}}M\con M$ (the last is possible by the GCH assumption).
	For every $\kappa\in K$ let $\gamma_{\kappa}=\sup M\cap\kappa$.
	$\cof{\gamma_\kappa}=\bar{\mu}<\gamma_\kappa$,
	as $\psu{<\bar{\mu}}M \con M$ implies that $\left| M \cap \kappa \right| = \bar{\mu}$.
	In particular $\gamma_{\kappa}$ is singular, and we denote $C_{\gamma_{\kappa}}^{M}:=C_{\gamma_{\kappa}}\cap M$
	(note the slight abuse of notation --
	we think of $C_{\gamma_{\kappa}}^{M}$ as ``the $C_{\gamma_{\kappa}}$ of $M$'' even though $\gamma_{\kappa}$ and $C_{\gamma_{\kappa}}$ are not in $M$).
	This is unbounded and ${<}\bar{\mu}$-closed in $\gamma_{\kappa}$.

	We now inductively define a matrix $\left\langle M_{\alpha,\kappa}\mid\left\langle \alpha,\kappa\right\rangle \in\mu\times(\{0\}\cup K)\right\rangle \in M$
	of elementary submodels of $M$. The order of the induction is the
	left-lexicographic order on $\mu\times(\{0\}\cup K)$ -- for $\alpha\in\mu$
	we first define $\left\langle M_{\alpha,\kappa}\mid\kappa\in\{0\}\cup K\right\rangle $
	before moving to $\alpha+1$. $M_{0,0}$ is the closure of $K$ in
	$M$ (so of size $\left|K\right|<\theta^{+}\leq\bar{\mu}$). Assume
	that for some $\left\langle \alpha,\kappa\right\rangle \in\mu\times K$,
	$\left\langle M_{\alpha,\kappa'}\mid\kappa'\in\{0\}\cup(K\cap\kappa)\right\rangle $
	is defined, is an element of $M$, and is such that each $M_{\alpha,\kappa'}$
	is of size $<\bar{\mu}$ (in $M$). So $\left|\bigcup_{\kappa'}M_{\alpha,\kappa'}\right|<\bar{\mu}$,
	and in particular $\left(\bigcup_{\kappa'}M_{\alpha,\kappa'}\right)\cap\kappa=\left(\bigcup_{\kappa'}M_{\alpha,\kappa'}\right)\cap\gamma_{\kappa}$
	is bounded in $\gamma_{\kappa}$ (whose cofinality is $\bar{\mu}$),
	say by $\delta$. $C_{\gamma_{\kappa}}^{M}$ is unbounded in $\gamma_{\kappa}$,
	so there is some $\xi_{\alpha,\kappa}\in C_{\gamma_{\kappa}}^{M}\smin\delta$.
	If $\otp(C_{\gamma_{\kappa}})\leq\nu_{\kappa}$ then we pick the least such $\xi_{\alpha,\kappa}$.
	Otherwise, since $C_{\gamma_{\kappa}}^{M}$ must also be of cofinality $\bar{\mu}>\omega=\cof{\nu_{\kappa}}$ ($\nu_{\kappa}$ is the order-type of a club in some element of $S_{\kappa}\con E_{\omega}^{\kappa}$),
	there must be cofinally many elements of $C_{\gamma_{\kappa}}^{M}$
	above the $\nu_{\kappa}$ place of $C_{\gamma_{\kappa}}$, so we can
	choose the least of those. Then we set $M_{\alpha,\kappa}$ as the
	closure of
	\[
		\bigcup_{\kappa'}M_{\alpha,\kappa'}\cup\left\{ \sup\Big(\bigcup_{\kappa'}M_{\alpha,\kappa'}\Big)\cap\kappa,\,\xi_{\alpha,\kappa},\left\langle M_{\alpha,\kappa'}\mid\kappa'\in\{0\}\cup(K\cap\kappa)\right\rangle \right\}
	\]
	in $M$. Still, we remain of size $<\bar{\mu}$.

	If $M_{\alpha,\kappa}$ is defined for all $\kappa\in\{0\}\cup K$
	then $M_{\alpha+1,0}=\bigcup_{\kappa\in K}M_{\alpha,\kappa}$, and
	if $M_{\beta,0}$ is defined for all $\beta<\alpha$ for limit $\alpha$,
	then $M_{\alpha,0}=\bigcup_{\beta<\alpha}M_{\beta,0}$. We also define
	$M_{\mu,0}=\bigcup_{\beta<\mu}M_{\beta,0}$.

	We claim that $\left\langle M_{\alpha,0}\mid\alpha\in\mathrm{Lim}(\mu+1)\right\rangle $
	is as required. It is internally approachable as for limit $\alpha$,
	$\left\langle M_{\gamma,0}\mid\gamma\in\mathrm{Lim}(\alpha+1)\right\rangle \in M_{\alpha,1}\con M_{\alpha+\omega,0}$.
	It is continuous by definition. Let $\kappa\in K$, we claim that
	$\{\xi_{\beta,\kappa}\mid\beta<\alpha\}$ is unbounded in $M_{\alpha,0}\cap\kappa$.
	If $\zeta\in M_{\alpha,0}\cap\kappa$ then there is $\beta<\alpha$
	such that $\zeta\in M_{\beta,\kappa}\cap\kappa$. Then in fact $\zeta\in M_{\beta,\kappa}\cap\gamma_{\kappa}$,
	and by our choice we have $\xi_{\beta+1,\kappa}>\zeta$ as required.
	So $\sup M_{\alpha,0}\cap\kappa=\sup\{\xi_{\beta,\kappa}\mid\beta<\alpha\}$
	and since $\{\xi_{\beta,\kappa}\mid\beta<\alpha\}\con C_{\gamma_{\kappa}}$
	and $\alpha<\mu\leq\bar{\mu}=\cof{C_{\gamma_{\kappa}}}$, this is
	also an element of $C_{\gamma_{\kappa}}$, and by our choices its
	place in $C_{\gamma_{\kappa}}$ is not $\nu_{\kappa}$, so as we noted
	before, it is an element of $T_{\kappa}$.
\end{proof}

\subsubsection{\label{subsec:Forcing-non-reflecting-stationar}Forcing non-reflecting
	stationary sets}

Recall that a stationary set $S\con\kappa$ is said to \emph{reflect
	at $\lambda<\kappa$} if $S\cap\lambda$ is stationary at $\lambda$.
We say that $S$ is \emph{non-reflecting }if it does not reflect at
any $\lambda<\kappa$.
\begin{definition}
	Let $\kappa$ be a regular uncountable cardinal. The poset $\nr{\kappa}$
	adding a Non-Reflecting stationary subset of $\kappa$ consists of
	conditions which are functions $p:\alpha+1\to\{0,1\}$ where $\alpha<\kappa$
	such that $p(\beta)=0\Rightarrow\cof{\beta}=\omega$ and for every
	limit $\mu<\kappa$ of uncountable cofinality, $p^{-1}\{0\}\cap\mu$
	is \emph{not} stationary at $\mu$. A condition $p$ is stronger than
	$q$ if $q\con p$. We denote $p^{-1}\{i\}$ by $S_{p,i}$.
\end{definition}

\begin{proposition}
	\label{prop:NR-props} Let $\kappa$ be a regular uncountable cardinal.
	Then:
	\begin{enumerate}
		\item $\nr{\kappa}$ is $\sigma$-closed.
		\item $\nr{\kappa}$ is ${<}\kappa$ strategically-closed.
		\item If $G\con\nr{\kappa}$ is generic then in $V[G]$ the set
		      \[
			      S_{G}=\bigcup_{p\in G}S_{p,0}=\left\{ \gamma<\kappa\mid\exists p\in G\,p(\gamma)=0\right\}
		      \]
		      is a non-reflecting stationary subset of $\kappa$.
	\end{enumerate}
\end{proposition}

\begin{proof}
	(1) If $\left\langle p_{n}\mid n<\omega\right\rangle $ is a decreasing sequence of conditions then
	\[
		\bigcup_{n<\omega} p_{n}\cup\left\{ \left\langle \sup_{n<\omega} \dom(p_{n}),1\right\rangle \right\}
	\]
	is also a condition since the non-reflection condition regards only
	limits of uncountable cofinality.

	(2) Let $\nu<\kappa$ and fix some club $C\con\kappa$. We define
	a strategy for player II in the game of length $\nu$, where the players take turns in choosing a decreasing sequence of conditions, and limit stages are chosen by player II.
	We proceed  inductively, making sure that at every limit stage $\alpha$, if the game produced
	$\left\langle p_{\gamma}\mid\gamma<\alpha\right\rangle $ then $C\cap\bigcup\{S_{p_{\gamma},1}\mid\gamma<\alpha\}$
	is club in $\sup\{\dom(p_{\gamma})\mid\gamma<\alpha\}$. At successor
	stages, if player I chooses $p$, let $\alpha=\min C\smin\dom(p)$
	and player II chooses $p'=p\cup\{\left\langle \gamma,1\right\rangle \mid\gamma\in\alpha+1\smin\dom(p)\}$.
	In particular $\alpha=\max\dom(p')\in C\cap S_{p',1}$. This is indeed
	a condition since $S_{p',0}=S_{p,0}$ so the non-reflection requirement
	holds. At limit stages, if $\left\langle p_{\gamma}\mid\gamma<\alpha\right\rangle $
	is defined by the strategy, let $\delta=\sup\{\dom(p_{\gamma})\mid\gamma<\alpha\}$,
	then the successor stages ensure that $C\cap\bigcup\{S_{p_{\gamma},1}\mid\gamma<\alpha\}$
	is cofinal in $\delta$ and the induction assumption ensures it is
	closed. So player II chooses $p_{\alpha}=\bigcup_{\gamma<\alpha}p_{\gamma}\cup\left\{ \left\langle \delta,1\right\rangle \right\} $,
	which ensures that the inductive assumptions remain true.
	To see that this is a condition, if $\mu<\delta$ then $S_{p_{\alpha},0} \cap \mu =S_{p_{\gamma},0} \cap \mu$
	for some $\gamma<\alpha$ so $S_{p_{\alpha},0}\cap\mu$ is not stationary
	by assumption. If $\mu>\delta$ then $S_{p_{\alpha},0}$ is bounded
	below $\mu$ so surely not stationary. For $\delta$, the assumption
	was that $C\cap\bigcup\{S_{p_{\gamma},1}\mid\gamma<\alpha\}$ is club
	in $\delta$, and it is surely disjoint from $S_{p_{\alpha},0}$.

	(3) For every $\alpha<\kappa$ $S_{G}\cap\alpha=S_{p,0}\cap\alpha$
	for some $p$, so is not stationary at $\alpha$. Hence we only need
	to show $S_{G}$ is stationary in $\kappa$.
	Let $\dot{C}$ be a name for a club in $\kappa$.
	We need to show that
	\[
		D=\left\{ p\in\nr{\kappa}\mid\exists\gamma\in\dom(p)\,(p(\gamma)=0\land p\Vdash\check{\gamma}\in\dot{C})\right\}
	\]
	is dense.
	Let $p\in\nr{\kappa}$. 
	Define a decreasing sequence of conditions $p_n$ and ordinals $\alpha_n$ such that  $p_0=p$,
	$p_{n+1} \Vdash \alpha_n \in \dot{C} \smin \dom(p_n)$
	and $\alpha_n \in \dom(p_{n+1})$.
	Then $\alpha^*:=\bigcup_n \dom(p_{n}) = \sup_{n<\omega} \alpha_n$ and like before,
	\[
		\bigcup_n p_{n}\cup\left\{ \left\langle \alpha^*,0\right\rangle \right\}
	\]
	is a condition, which forces that $\alpha^* \in \dot{C}_G \cap S_G$, i.e. is in $D$.
	So $S_{G}$ is indeed a non-reflecting stationary set in $V[G]$.
\end{proof}
\begin{definition}
	Let $K=\left\langle \kappa_{\zeta}\mid\zeta<\theta\right\rangle $
	be an increasing sequence of uncountable cardinals which are successors
	of regulars, $\theta<\kappa_{0}$. Let $\nr K$ be the full support
	product $\Pi_{\zeta<\theta}\nr{\kappa_{\zeta}}$.

	Note that by strategic closure of each component this is the same
	as the full support iteration, and is also ${<}\kappa_{0}$-strategically
	closed.
\end{definition}

\begin{theorem}
	Let $G\con\nr K$ be generic, for every $\zeta$ $G_{\zeta}$ the
	projection of $G$ on the $\zeta$th coordinate, $S_{\zeta}=S_{G_{\zeta}}$
	the induced stationary set and $T_{\zeta}:=\kappa_{\zeta}\smin S_{\zeta}=\bigcup_{p\in G}S_{p(\zeta),1}$.
	Then $\left\langle T_{\zeta}\mid\zeta<\theta\right\rangle $ are strongly
	mutually fat.
\end{theorem}

\begin{proof}
	Fix some $\nu<\kappa_{0}$ and let  $\mathfrak{A}=\left\langle \lambda,\in,\theta,K,\nr K^{V},...\right\rangle $
	be, in $V[G]$, an algebra  for some large enough $\lambda$
	(where $K,\nr K^{V}$ are coded in $V$ as ordinals)\footnote{We take the algebra to be on an ordinal so there are no ``new''
		elements of $\mathfrak{A}$. This doesn't limit generality.}. We will assume without loss of generality $\theta\leq\nu$.  We
	aim to show that the set of conditions forcing that there is a witness
	with respect to $\mathfrak{A}$ to the mutual fatness of $\left\langle T_{\zeta}\mid\zeta<\theta\right\rangle $,
	is dense.

	Given some $p_{0}\in\nr K$, we inductively define a decreasing sequence
	of conditions $\left\langle p_{\alpha}\mid\alpha\leq\nu\right\rangle $
	and an internally approachable continuous sequence of sets $\left\langle N_{\alpha}\mid0<\alpha\leq\nu\right\rangle \in V$
	such that $p_{\alpha}\in N_{\alpha+1}$ and for $\alpha>0$ $p_{\alpha}\Vdash\check{N}_{\alpha}\prec\mathfrak{A}$,
	as follows. There is some name $\dot{N}$ such that $p_{0}$ forces
	that $p_{0}\in\dot{N}\prec\mathfrak{A}$, $\theta+1\con\dot{N}$ and
	$|\dot{N}|<\kappa_{0}$ . By strategic closure there is some $p_{1}\leq p_{0}$
	and some $N_{1}\in V$ such that $p_{1}\Vdash\dot{N}=\check{N}_{1}$.
	If $p_{\alpha},N_{\alpha}$ are defined, we apply the same procedure
	to get $p_{\alpha+1}$ and $N_{\alpha+1}$, with the additional requirement
	that $N_{\alpha}\cup\left\{ \left\langle N_{\gamma}\mid0<\gamma\leq\alpha\right\rangle \right\} \con N_{\alpha+1}$.
	For $\alpha$ limit, set $N_{\alpha}=\bigcup_{\beta<\alpha}N_{\beta}$,
	and $p_{\alpha}$ is defined for every $\zeta<\theta$ by $p_{\alpha}(\zeta)=\bigcup_{\beta<\alpha}p_{\beta}(\zeta)\cup\left\{ \left\langle \bigcup_{\beta<\alpha}\dom\left(p_{\beta}(\zeta)\right),1\right\rangle \right\} $.
	Note that in the limit case, since for every $\beta<\alpha$ $p_{\beta}\in N_{\beta+1}\con N_{\alpha}$,
	for every $\zeta<\theta$ $\dom(p_{\beta}(\zeta))<\sup(N_{\alpha}\cap\kappa_{\zeta})$,
	so $\bigcup_{\beta<\alpha}\dom\left(p_{\beta}(\zeta)\right)\leq\sup(N_{\alpha}\cap\kappa_{\zeta})$,
	and we can actually make sure they are equal. So $\sup(N_{\alpha}\cap\kappa_{\zeta})\in S_{p_{\alpha}(\zeta),1}$.
	Also note that $p_{\alpha}$ forces that for every $\beta<\alpha$
	$N_{\beta}\prec\mathfrak{A}$ so it also forces that $N_{\alpha}\prec\mathfrak{A}$.

	To conclude, $p_{\nu}$ forces that the sequence $\left\langle N_{\alpha}\mid\alpha\leq\nu,\alpha\,\mathrm{limit}\right\rangle $
	will be an internally approachable sequence of elementary submodels
	of $\mathfrak{A}$ such that for every limit $\alpha\leq\nu$ and
	every $\zeta<\theta$, $\sup(N_{\alpha}\cap\kappa_{\zeta})\in S_{p_{\nu}(\zeta),1}\con T_{\zeta}$.
	In other words, $p_{\nu}\leq p_{0}$ forces that $\left\langle N_{\alpha}\mid\alpha\leq\nu,\alpha\,\mathrm{limit}\right\rangle $
	witnesses strong mutual fatness of $\langle \dot{T}_{\zeta}\mid\zeta<\theta\rangle $
	with respect to $\mathfrak{A}$, as required.
\end{proof}

\section{\label{sec:coding}Coding by shooting clubs}
Let $S \con E^\kappa_\omega$ for some regular $\kappa>\omega$,
and let $M$ be a transitive model (satisfying enough of $\zf$) such that  $\kappa,S \in M$.
As we have seen in Lemma \ref{lem:correspond}, $S$ is stationary iff the set $\{ s \in \cpow{\kappa} \mid \sup s \in S \} $ is stationary.
By Lemma \ref{lem:projections} this is stationary iff $\{ s \in \cpow{M} \mid  \sup s\cap \kappa \in S\} $ is stationary,
and this is iff $(M,\in) \vDash \stat s (\sup s \cap \kappa \in S)$.
Hence using $\LL(\aaa)$ we can identify whether a certain set of ordinals of cofinality  $\omega$ is stationary or not.
Now, if $\vec{S}=\left\langle S^{\alpha}\mid\alpha<\kappa\right\rangle \in C(\aaa)$ are all subsets of $E^\kappa_\omega$ for some $\kappa$,
then $\{ \alpha<\kappa \mid S^\alpha \text{ is stationary} \}$ is in $C(\aaa)$ as well.
So we think of such sequences $\vec{S}$ as giving us a way to ``code'' sets of ordinals into $C(\aaa)$ via their stationarity.
We will use this observation, together with the tools to destroy the stationarity of certain sets, to code sets into $C(\aaa)$ in generic extensions.

\begin{definition}
	Let $X$ be a set of ordinals and $\kappa>\sup X$ a successor of
	a regular cardinal. If $\vec{S}=\left\langle S^{\alpha}\mid\alpha<\kappa\right\rangle $
	is a partition of a co-fat set $\tilde{S}\con E_{\omega}^{\kappa}$
	into disjoint stationary sets, then we denote the poset $\des{\bigcup\left\{ S^{\alpha}\mid\alpha\in X\right\} }$
	by $\Der(X,\kappa,\vec{S})$. $\dot{\der}(X,\kappa,\vec{S})$ denotes
	the name for its generic, and $\der(X,\kappa,\vec{S})$ will denote
	an arbitrary generic. If $\kappa$ and $\vec{S}$ are fixed and clear
	from the context we will often write simply $\Der X$, $\dot{\der}X$ and $\der X$.

	In a generic extension by shooting a club through the complement of
	the stationary set  $\bigcup\left\{ S^{\alpha}\mid\alpha\in X\right\} $
	we get that
	\begin{align*}
		X & =\left\{ \alpha<\sup X\mid S^{\alpha}\cap\der X=\emp\right\}           \\
		  & =\left\{ \alpha<\sup X\mid S^{\alpha}\text{ is not stationary}\right\}
	\end{align*}
	hence, as we observed above, if $\vec{S}$ is in the $C(\aaa)$ of the extension (e.g. if
	$\vec{S}\in L$), then so will $X$ be. Thus we refer to $\Der(X,\kappa,\vec{S})$
	as ``coding $X$ into $C(\aaa)$ at $\kappa$, using $\vec{S}$\,''.  We make this formal in the
	following way:
\end{definition}

\begin{proposition}
	\label{cor:C(aa) of coding}Let $X\con\ord$, $\kappa>\sup X$ a successor
	of a regular cardinal, and assume there are $S\con E_{\omega}^{\kappa}$
	a co-fat stationary set, and $\vec{S}=\left\langle S^{\alpha}\mid\alpha<\kappa\right\rangle $
	a partition of $S$ into disjoint stationary sets such that $\vec{S}\in C(\aaa)^{V[\der(X,\kappa,\vec{S})]}$.
	Then $X\in C(\aaa)^{V[\der(X,\kappa,\vec{S})]}$.

	In particular, if $V=L[X]$ and $\vec{S}\in L$ then $C(\aaa)^{V[\der(X,\kappa,\vec{S})]}=L[X]$.
\end{proposition}

\begin{proof}
	As we noted $X=\left\{ \alpha<\sup X\mid S^{\alpha}\text{ is not stationary}\right\} \in C(\aaa)^{V[\der(X,\kappa,\vec{S})]}$.
	So $L[X]\con C(\aaa)^{V[\der(X,\kappa,\vec{S})]}$.
	If $\vec{S}\in L$ then by Theorem \ref{Thm:it-club-shooting-countable}(\ref*{enu:calculate-stat-count})(b)
	we have $C(\aaa)^{V[\der(X,\kappa,\vec{S})]}\con V$, so if $V=L[X]$
	we get our equality.
\end{proof}

\subsection{\label{subsec:Coding-into-V=00003DC(aa)}Coding a set into a model
	of $V=C(\protect\aaa)$}

In this section we follow the method of Zadro\.{z}ny in \cite{Zadrozny-IteratingOrdinalDefinability1983}
to code sets into a model satisfying ``$V=C(\aaa)$'' using iterated
club shooting.
\begin{theorem}
	\label{thm:coding}Let $V=L[A]$ for some set $A$ such
	that there are:
	\begin{enumerate}
		\item $\left\langle \kappa_{n}\mid n<\omega\right\rangle \in L$ an increasing sequence of successors of regular cardinals (of $V$) above $\kappa_{-1}=\sup A$
		      such that for every $n$ $2^{\kappa_{n-1}}<\kappa_{n}$.
		\item $\left\langle T_{n},\tilde{S}_{n}\mid n<\omega\right\rangle \in L$
		      such that for every $n$, $T_{n}\con E_{\omega}^{\kappa_{n}}$ is
		      stationary in $\kappa_{n}$ (in $V$), and $\tilde{S}_{n}\con E_{\omega}^{\kappa_{n}}\smin T_{n}$
		      is co-fat (again in $V$). It will also be convenient to assume that
		      $\tilde{S}_{n}\con\kappa_{n}\smin\kappa_{n-1}$.
		\item $\left\langle \vec{S}_{n}=\left\langle S_{n}^{\alpha}\mid\alpha<\kappa_{n}\right\rangle \mid n<\omega\right\rangle \in L$
		      such that for every $n$ $\vec{S}_{n}$ is a partition of $\tilde{S}_{n}$
		      into disjoint stationary sets.
	\end{enumerate}
	Then there is a forcing extension of $V$ satisfying ``$V=C(\aaa)$''.

\end{theorem}

\begin{remark}
	\begin{enumerate}
		\item Stationarity and fatness of the sets above are with respect to $V$
		      even though the sets are from $L$.
		\item The assumptions hold for any $A$ which is $L$-generic. See Proposition
		      \ref{cor:coding tools}.
		\item $L$ can be replaced by other canonical inner models which are provably
		      contained in $C(\aaa)$, such as the Dodd-Jensen core model (see \cite{IMEL2}).
	\end{enumerate}
\end{remark}
\begin{proof}
	We define inductively an iteration of club shooting forcings as follows.
	Set $\Der^{0}A=\left\{ 1\right\} $ and $\dot{\der}^{0}A=\check{A}$ (which
	is also considered as a $\Der^{0}A$-name). If we've inductively
	defined $\Der^{n}A$ and $\dot{\der}^{n}A$ as a $\Der^{n}A$-name
	for a subset of $\kappa_{n}$, then we let $\dot{S}_{n}$ be a $\Der^{n}A$-name
	for $\bigcup\left\{ S_{n}^{\alpha}\mid\alpha\in\dot{\der}^{n}A\right\} $,
	\[
		\Der^{n+1}A:=\Der^{n}A*\des{\dot{S}_{n}}=\Der^{n}A*\dot{\Der}(\dot{\der}^{n}A,\kappa_{n},\vec{S}_{n})^{.}
	\]
	and $\dot{\der}^{n+1}A=\dot{\der}(\dot{\der}^{n}A,\kappa_{n},\vec{S}_{n})$
	-- the name for the generic club forced by $\des{\dot{S}_{n}}$.
	Let $\Der^{*}A$ be the full support limit of the iteration and note
	that it satisfies the assumptions of Theorem \ref{Thm:it-club-shooting-countable},
	since at each stage we destroy sets disjoint from the $T_{n}$s.

	So, at stage $1$ we shoot a club $\der^{1}A$ through the complement
	of $\bigcup\left\{ S_{0}^{\alpha}\mid\alpha\in A\right\} $, thus
	destroying the stationarity of exactly these sets out of $\vec{S}_{0}$,
	so  we code $A$ into $C(\aaa)$. At stage $2$ we shoot a club through
	the complement of $\bigcup\left\{ S_{1}^{\alpha}\mid\alpha\in\der^{1}A\right\} $
	thus coding $\der^{1}A$ into $C(\aaa)$, and so on. After $\omega$
	many steps we catch our tail, so that the entire generic of $\Der^{*}A$
	is coded into $C(\aaa)$
	(Theorem \ref{Thm:it-club-shooting-countable}
	is used to show that what we coded at a certain stage won't be destroyed at subsequent stages).
	Let's see this formally.

	Let $G\con\Der^{*}A$ be generic. We claim that $G\in C(\aaa)^{V[G]}$
	so $V[G]=C(\aaa)^{V[G]}$.
	For every $n$ let $\der^{n}A=\big(\dot{\der}^{n}A\big )^{G}$, and it is clear from the construction that $G$ can be obtained from $\langle \der^{n}A\mid n<\omega\rangle $,
	so we need to show that this sequence is in $C(\aaa)$.
	For every $n$ and $\alpha<\kappa_{n}$, if $\alpha\in\der^{n}A$ then by the properties of the club shooting $S_{n}^{\alpha}$ is not stationary in $V[G]$,
	while if $\alpha\notin\der^{n}A$, by Theorem \ref{Thm:it-club-shooting-countable} $S_{n}^{\alpha}$ \emph{is }stationary in $V[G]$.
	So $\der^{n}A=\{ \alpha<\kappa_{n}\mid V[G]\vDash  S_{n}^{\alpha}\,\text{ is stationary}\} $
	which is in $C(\aaa)^{V[G]}$ since $\vec{S}_{n}\in L$.
	Since $\langle \kappa_{n},T_{n},\vec{S}_{n}\mid n<\omega\rangle \in L$,
	also the \emph{sequence }$\langle \der^{n}A\mid n<\omega\rangle $
	is in $C(\aaa)^{V[G]}$.
\end{proof}
\begin{proposition}
	\label{cor:coding tools}If $V=L[A]$ where $A$ is set-forcing
	generic over $L$ then the assumptions of Theorem \ref{thm:coding}
	hold.
\end{proposition}

\begin{proof}
	In this case, for large enough cardinals, $V$ agrees with $L$ on
	cardinalities and the notions of stationarity and fatness, and we
	have for large enough $\kappa>\sup A$ $2^{\kappa}=\kappa^{+}$ and
	$\square_{\kappa}$. So we can pick $\left\langle \kappa_{n}\mid n<\omega\right\rangle \in L$
	an increasing sequence of cardinals (of $V$) above $\kappa_{-1}=\sup A$
	which are successors of regular cardinals and such that for every
	$n$ $2^{\kappa_{n-1}}<\kappa_{n}$ and $\square_{\kappa_{n}^{-}}$
	holds (where $\kappa^{-}$ denotes the cardinal predecessor of a successor cardinal $\kappa$).
	Then for each $n$ we split $E_{\omega}^{\kappa_{n}}$
	into two disjoint stationary sets, take $T_{n}$ as one of them and
	apply Lemma \ref{lem:square-implis-cofat} to the other to obtain
	$\tilde{S}_{n}$, and then partition it into disjoint sets. All of
	this is done in $L$ but the desired properties are retained in $V$.
\end{proof}

In Theorem \ref{thm:coding} we only used a sequence of co-fat stationary
sets, but if we add the assumption that their complements form a \emph{strongly
	mutually fat }sequence, we can get a better result:
\begin{theorem}
	Let $V=L[A]$ for some set $A$ such that the assumptions
	of Theorem \ref{thm:coding} hold, and further assume that $\left\langle T_{n}\mid n<\omega\right\rangle $
	is strongly mutually fat. Then there is a forcing extension $W$
	of $V$ such that $C(\aaa)^{W}=W$ and $H(\kappa_{0})^{W}=H(\kappa_{0})^{V}$.
\end{theorem}

\begin{proof}
	We apply the same proof as of Theorem \ref{thm:coding}, but now by
	Theorem \ref{Thm:it-club-shooting-full}.\ref*{enu:distributivity-full}
	the forcing is ${<}\kappa_{0}$ distributive, so we get $H(\kappa_{0})^{W}=H(\kappa_{0})^{V}$.
\end{proof}

This means that any ``local'' statement that can be forced over
$L$, is consistent with $V=C(\aaa)$, where ``local'' is in fact
$\Sigma_{2}$ (cf. \cite[pg. 86]{Solovay.Reinhardt.ea-StrongAxiomsInfinity1978}, where the following is stated with $H(\theta)$ replaced by $V_\theta$ such that $\left| V_\theta \right| = \theta$):
\begin{fact}
	$\Phi$ is a $\Sigma_{2}$ sentence iff $\vdash\Phi\leftrightarrow\exists\theta(H(\theta)\vDash\Phi)$.
\end{fact}


\begin{corollary}
	If $\Phi$ is a $\Sigma_{2}$ statement, perhaps with ordinal parameters,
	which is forceable over $L$, then there is a forcing extension of
	$L$ satisfying $V=C(\aaa)+\Phi$.
\end{corollary}

\begin{proof}
	If $L[A]$ is a forcing extension of $L$ satisfying $\Phi$, then
	we can choose a sequence $K=\left\{ \kappa_{n}\mid n<\omega\right\} $
	large enough (above the $H(\theta)$ which satisfies $\Phi$) so that
	$L$ and $L[A]$ agree on all relevant notions, use Theorem \ref{thm:mut-fat-squar}
	to obtain  a sequence $\left\langle \bar{S}_{\kappa}\mid\kappa\in K\right\rangle $
	such that for every $\kappa\in K$ $\bar{S}_{\kappa}\con S_{\kappa}$
	is stationary and the sequence $\left\langle T_{\kappa}\mid\kappa\in K\right\rangle $
	where $T_{\kappa}=\kappa\smin\bar{S}_{\kappa}$ is strongly mutually
	fat. Then applying the previous theorem we get the desired extension.
\end{proof}
This gives us, for example, that for every $\kappa$ of uncountable
cofinality, $Con(\zfc)$ implies the consistency of $V=C(\aaa)+2^{\aleph_{0}}=\kappa$.
This is in stark contrast to the case of $C^{*}$, where $V=C^{*}$
implies that $2^{\aleph_{0}}\in\left\{ \aleph_{1},\aleph_{2}\right\} $
and for any $\kappa>\aleph_{0}$ $2^{\kappa}=\kappa^{+}$ (cf. \cite[Corollary to Theroem 5.20]{IMEL}).

\subsection{\label{subsec:Iterating-C(aa)}Iterating $C(\protect\aaa)$ }

In this section we want to use the coding construction to produce
models with decreasing iterations of $C(\aaa)$. First we note that
the construction above gives us finite iterations:
\begin{proposition}
	\label{prop:finite-descending}In the construction of Theorem \ref{thm:coding},
	for every $n<\omega$ $C(\aaa)^{L\left[\der^{n+1}A\right]}=L\left[\der^{n}A\right]$.
\end{proposition}

\begin{proof}
	Proposition \ref{cor:C(aa) of coding}.
\end{proof}
So at each finite step of the iteration, where we ``code $A$ $n$
times'', we get a decreasing sequence of $C(\aaa)$s of length $n$,
but after $\omega$ stages we don't get a decreasing sequence of length
$\omega$, but rather ``catch our tail'' and get ``$V=C(\aaa)$''
(this is the reason we denoted the final iteration above by $\Der^{*}A$
and not by $\Der^{\omega}A$, which we will use shortly). To get
an infinite decreasing sequence, we need to make sure that on each
step we lose something, but we still have infinitely many things to
lose. This is accomplished (also following Zadro\.{z}ny) by adding
$\omega$ many ``partial codings'' -- for each $n$ we code a set
$A_{n}$ ``$n$ many times'', so that taking $C(\aaa)$ will drop
the last coding of each $A_{n}$, giving us $n-1$ codings of $A_{n}$.
We will eventually lose all codings for $A_{n}$ after $n$ stages,
but we still have infinitely more $A_{m}$s for which there are more
codings to lose. In fact, the $A_{n}$s will actually be the same
set, which we code at different places so the codings are different.

\begin{theorem}
	\label{thm:descending-omega}Let $V=L[A]$ for some set
	$A$ such that the assumptions of Theorem \ref{thm:coding} hold.
	Then there is a forcing extension of $V$ satisfying: for every $n<\omega$
	$C(\aaa)^{n}\ne C(\aaa)^{n+1}$, $C(\aaa)^{\omega}\vDash\zfc$, and $C(\aaa)^{\omega+1}=C(\aaa)^{V}$.
\end{theorem}

\begin{proof}
	Given the sequences $\langle \kappa_{n},T_{n},\tilde{S}_{n},\vec{S}_{n}\mid n<\omega\rangle \in L$
	as in the assumptions,	we re-index them as
	\[
		Z = \langle \kappa_{n,i},T_{n,i},\tilde{S}_{n,i},\vec{S}_{n,i}\mid i\leq n<\omega\rangle \in L
	\]
	so that $\kappa_{n,i}<\kappa_{m,j}$ iff $\left(n,i\right)<_{\lex}\left(m,j\right)$.
	Note that this is in fact a well order of order-type $\omega$. Let
	$\kappa_{\omega}=\sup\left\{ \kappa_{n,i}\mid n<\omega,i\leq n\right\} $.
	For the sake of clarity we will now describe the forcings we use in
	a bit less formal way, not using names for the forcings or the sets
	added. We think of it as inductively defining the iterations in the
	corresponding extensions.

	For every $n$, denote $\der^{n,-1}A=A$ and for each $n$ and $0\leq i\leq n$
	inductively define
	\begin{align*}
		\Der^{n,i}A & :=\Der^{n,i-1}A*\Der(\der^{n,i-1}A,\kappa_{n,i},\vec{S}_{n,i}).
	\end{align*}
	Denote $\Der^{n}A=\Der^{n,n}A$, and let $\Der^{\omega}A$ be the
	full support \emph{product} of the $\Der^{n}A$s. Note that this can
	be viewed also as a countable iteration of club shooting forcings,
	so Theorem \ref{Thm:it-club-shooting-countable} applies. Denote
	the generic extension $L\left[\left\langle \der^{n,i}A\mid i\leq n<\omega\right\rangle \right]$
	by $W$. The idea is as follows -- for each $n$, $\Der^{n}A$ codes
	the set $A$ $n+1$ many times -- at the cardinals $\kappa_{n,0}\till\kappa_{n,n}$.
	An important note is that for $n\ne m$, $\Der^{n}A$ and $\Der^{m}A$
	are forced independently. So, when we take the $C(\aaa)$ of the extension,
	the last coding at each ``level'' will drop (see Figure \ref{fig:descending-omega}).
	But since we have infinitely many levels, we can repeat this procedure
	infinitely many times, and only then will all the codings drop.

	\begin{figure}
		\centerline{
			\begin{tikzpicture}
				\draw(-0.8,9.7) node{$V[G]=$};
				\foreach \t in {0,...,3} 
					{
						\def\z{\t*3} 
						\draw (\z,0) -- (\z,8);
						\draw [dashed] (\z,8 ) -- (\z,10);
						\draw (\z,9) node[left]{$C(\mathtt{aa})^{\t}$};
						\foreach \x in {0,...,3}
							{
								\foreach \y in {0,...,\x}
									{
										\draw (\z-0.1,2*\x+\y/2) -- (\z+0.1,2*\x+\y/2);
										\draw (\z,2*\x+\y/2) node[left]{$\kappa_{\x,\y}$};
										\draw (\z,2*\x+\y/2) node[right]{$\partial^{\x,\y}A$};
										\pgfmathparse{\x-\t}
										\ifnum \y>\pgfmathresult \relax
											\draw[red,thick] (\z+0.1,2*\x+\y/2-0.1) -- (\z+1.1,2*\x+\y/2+0.1);
										\fi
									}
							}
					}

				\draw[dotted,very thick](10,9) -- (11.3,9);
				\draw[dotted,very thick](10.6,5) -- (11.7,5);

				\def\t{4}
				\def\z{\t*3+1}
				\draw (\z,0) -- (\z,8);
				\draw [dashed] (\z,8 ) -- (\z,10);
				\draw (\z,9) node[left]{$C(\mathtt{aa})^{\omega}$};
				\foreach \x in {0,...,3} 
					{
						\foreach \y in {0,...,\x} 
							{
								\draw (\z-0.1,2*\x+\y/2) -- (\z+0.1,2*\x+\y/2);
								\draw (\z,2*\x+\y/2) node[left]{$\kappa_{\x,\y}$};
								\draw (\z,2*\x+\y/2) node[right]{$\partial^{\x,\y}A$};
								\pgfmathparse{\x-\t}
								\ifnum \y>\pgfmathresult \relax
									\draw[red,thick] (\z+0.1,2*\x+\y/2-0.1) -- (\z+1.1,2*\x+\y/2+0.1);
								\fi
							}
					}
			\end{tikzpicture}
		}

		\caption{\label{fig:descending-omega}A descending $C(\protect\aaa)$ sequence
			of length $\omega$}
	\end{figure}
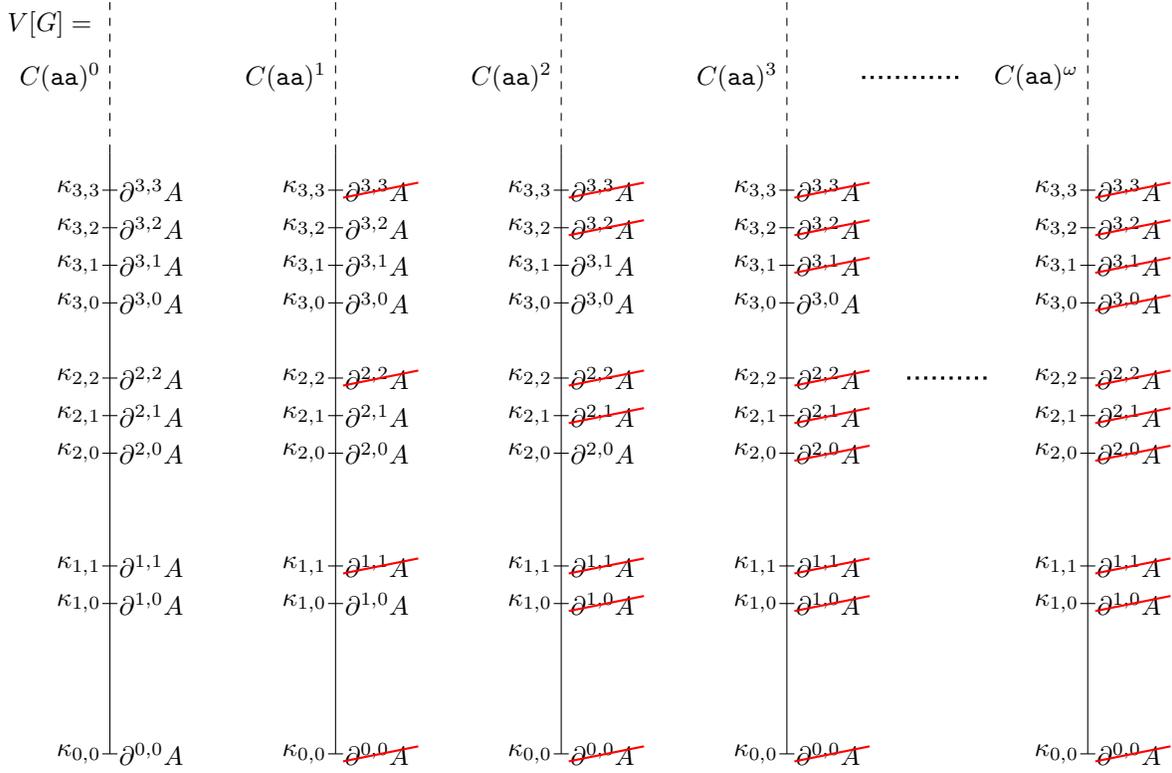

	Denote for every $k$ $W^{k}:=L\left[\left\langle \der^{n,i}A\mid k\leq n<\omega,i\leq n-k\right\rangle \right]$.
	\begin{claim}
		For every $k$, $C(\aaa)^{W^{k}}=W^{k+1}$.
	\end{claim}

	\begin{proof}
		$W^{k+1}=L\left[\left\langle \der^{n,i}A\mid k+1\leq n<\omega,i\leq n-k-1\right\rangle \right]$.
		So
		\[
			W^{k}=L\left[\left\langle \der^{n,i}A\mid k\leq n<\omega,i\leq n-k\right\rangle \right]=
			W^{k+1}\left[\left\langle \der^{n,n-k}A\mid k\leq n<\omega\right\rangle \right]
		\]
		where $Y:=\left\langle \der^{n,n-k}A\mid k\leq n<\omega\right\rangle $
		is generic over $W^{k+1}$ for an iteration of club shooting forcings
		of length $\omega$, destroying sets which are in $W^{k+1}$ (as they
		are computed from the previous stages). So we can apply Theorem \ref{Thm:it-club-shooting-countable}(\ref*{enu:calculate-stat-count})(b)
		to get that $C(\aaa)^{W^{k}}=C(\aaa)^{W^{k+1}[Y]}\con W^{k+1}$.
		On the other hand, as we discussed before, for every $n,i$ such that $k\leq n<\omega$ and $i<n-k-1$, $\der^{n,i+1}A \in W_k$ implies that
		$\der^{n,i}A\in C(\aaa)^{W^{k}}$,
		and using the entire sequence $Z \in L$ we get that also
		$\left\langle \der^{n,i}A\mid k+1\leq n<\omega,i\leq n-k-1\right\rangle \in C(\aaa)^{W_k}$.
		So we get equality.
	\end{proof}
	Thus we get that for every $k$, $\left(C(\aaa)^{k}\right)^{W}=L\left[\left\langle \der^{n,i}A\mid k\leq n<\omega,i\leq n-k\right\rangle \right]$
	and we get a descending sequence. Note that for every $k$, $A\in\left(C(\aaa)^{k}\right)^{W}$.

	Now we wish to analyze
	\[
		\left(C(\aaa)^{\omega}\right)^{W}=\bigcap_{k<\omega}\left(C(\aaa)^{k}\right)^{W}=\bigcap_{k<\omega}L\left[\left\langle \der^{n,i}A\mid k\leq n<\omega,i\leq n-k\right\rangle \right].
	\]
	For every $k$, denote by
	\[
		Q_{k}=\prod_{k\leq n<\omega}\bigast_{i\leq n-k}\dot{\Der}^{n,i}A
	\]
	the poset adding the generic object $\left\langle \der^{n,i}A\mid k\leq n<\omega,i\leq n-k\right\rangle $
	and by $B_{k}$ the complete Boolean algebra corresponding to it.
	We can without loss of generality assume that $B_{k+1}\con B_{k}$
	and clearly if $G$ is the generic for $B=B_{0}=\mathrm{ro}(\Der^{\omega}A)$,
	and $G_{k}=G\cap B_{k}$, then
	\[
		L[A][G_{k}]=L\left[\left\langle \der^{n,i}A\mid k\leq n<\omega,i\leq n-k\right\rangle \right].
	\]
	So, denoting $B_{\omega}=\bigcap_{k<\omega}B_{k}$, $G_{\omega}=G\cap B_{\omega}$,
	and using Fact \ref{fact:descending-generic}, we get that
	\[
		\left(C(\aaa)^{\omega}\right)^{W}=L[A][G_{\omega}]
	\]
	and in particular it satisfies $\zfc$, since $Q_{0}=\Der^{\omega}A$
	(and so also $B_{0}$) is $\omega$-distributive.

	To show that $\left(C(\aaa)^{\omega+1}\right)^{W}=C(\aaa)^{V}$ we
	want to show that $\left(C(\aaa)^{\omega}\right)^{W}$ agrees with
	$V$ on the notion of stationarity. For every $n<\omega$ and $k<l\leq n$,
	let
	\[
		g_{n}^{kl}:\bigast_{i\leq n-l}\dot{\Der}^{n,i}A\to\bigast_{i\leq n-k}\dot{\Der}^{n,i}A
	\]
	be the natural projection. Then this induces an $\omega$-normal
	system $f_{kl}:Q_{k}\to Q_{l}$ as in Lemma \ref{lem:kappa-closed},
	so we get that $\meq{Q_{0}}$ is $\sigma$-closed.  So by Lemma
	\ref{lem:sigma-closed}, for any $\lambda$ and $T\con\cpow{\lambda}$,
	$T\in V$, $T$ is stationary in $V$ iff it is stationary in $V[G_{\omega}]=\left(C(\aaa)^{\omega}\right)^{W}$.

	Now we can inductively prove that the stages of construction of $C(\aaa)^{V}$
	and $\left(C(\aaa)^{\omega+1}\right)^{W}$ are exactly the same, since
	in the successor step we always consider the stationarity of sets
	which are, by the induction hypothesis, sets in $V$, and this notion is the same in $V$ and in $\left(C(\aaa)^{\omega}\right)^{W}$.
	So indeed $\left(C(\aaa)^{\omega+1}\right)^{W}=C(\aaa)^{V}$.
\end{proof}

Note that if we apply this theorem to $V=L[\der A]$ instead
of $L[A]$, then we get that $\left(C(\aaa)^{\omega+1}\right)^{V}=C(\aaa)^{L[\der A]}=L[A]$.
This suggests that we can get longer  iterations, however to go past countable iterations
we would need the stronger distributivity properties provided by mutually
fat sets.
\begin{theorem}
	\label{thm:descending-long}Let $V=L[A]$ where $A$ is set-generic
	over $L$. Then for any ordinal $\delta$ there is a forcing extension
	of $V$ satisfying $\forall\alpha\leq\delta$ $C(\aaa)^{\alpha}\vDash\zfc$,
	$C(\aaa)^{\alpha}\ne C(\aaa)^{\alpha+1}$ and $C(\aaa)^{\delta+1}=L[A]$.
\end{theorem}

\begin{proof}
	As $V$ is a set-generic extension of $L$, there is some cardinal
	$\psi$ above which $V$ agrees with $L$ on cardinals and on the
	notion of stationarity, so in particular $\mathrm{GCH}$ and existence of global
	square hold above $\psi$. We assume this $\psi$ is $>\sup A \cup \delta$.
	Set $\theta=\psi^{+}$and fix sequences:
	\begin{enumerate}
		\item $K=\left\langle \kappa_{\eta}\mid\eta<\theta\right\rangle \in L$
		      an increasing sequence of successors of regular cardinals (of $V$)
		      above $\kappa_{-1}=\psi^{++}$  such that for every $\eta$ $2^{\kappa_{\eta-1}}<\kappa_{\eta}$.
		\item $\left\langle \bar{S}_{\eta}\mid\eta<\theta\right\rangle \in L$ a
		      sequence of stationary sets obtained as in Theorem \ref{thm:mut-fat-squar},
		      so for every $\eta$, $S_{\eta}\con E_{\omega}^{\kappa_{\eta}}$ is
		      stationary (in $V$), and the sequence $\left\langle T_{\eta}\mid\eta<\theta\right\rangle $
		      where $T_{\eta}=\kappa_{\eta}\smin\bar{S}_{\eta}$ is strongly mutually
		      fat (again in $V$).
		      It will also be convenient to assume that $\bar{S}_{\eta}\con\kappa_{\eta}\smin\kappa_{\eta}^{-}$.
		\item $\left\langle \vec{S}_{\eta}=\left\langle S_{\eta}^{\alpha}\mid\alpha<\kappa_{\eta}\right\rangle \mid\eta<\theta\right\rangle \in L$
		      such that for every $\eta$ $\vec{S}_{\eta}$ is a partition of $\bar{S}_{\eta}$
		      into disjoint stationary sets.
	\end{enumerate}
	These will be our ``coding tools''. We prove by induction on $\delta<\psi$
	that for every relevant $X$ (where which $X$s are ``relevant''
	is inductively defined as those $X$s which are used in the construction
	of previous stages) and any $\eta<\theta$ there is a notion of forcing,
	denoted $\Der^{\delta}(X,\eta)$, which is a full support iteration
	of club shooting forcings using cardinals from $\left\langle \kappa_{\alpha}\mid\alpha\in[\eta,\eta+1+\delta)\right\rangle $,
	such that the following holds:
	\begin{quote}
		\textbf{IH} If $Y$ is set-generic over $L$ such that $L[Y]$ agrees
		with $L$ on cardinalities and stationarity in the segment $\left\langle \kappa_{\alpha}\mid\alpha\in[\eta,\eta+1+\delta)\right\rangle $,
		and $\der^{\delta}(X,\eta)$ is generic over $L[Y]$, then in $L[Y][\der^{\delta}(X,\eta)]$
		the $C(\aaa)$-sequence has length at least $\delta+1$ and
		\[
			\left.C(\aaa)^{\delta+1}\right.^{L[Y][\der^{\delta}(X,\eta)]}=\left.C(\aaa)\right.^{L[Y]}.
		\]
	\end{quote}
	\begin{remark}
		\label{rem:larger iter}Note that if the above holds, and $\delta'=\delta+1+\gamma<\psi$ (we don't require
		anything on $\gamma$), then
		\begin{align*}
			\left(C(\aaa)^{\delta'}\right)^{L[Y][\der^{\delta}(X,\eta)]} 
        & =\left(C(\aaa)^{\delta+1+\gamma}\right)^{L[Y][\der^{\delta}(X,\eta)]}\\
			  & =\left(C(\aaa)^{\gamma}\right)^{\left(C(\aaa)^{\delta+1}\right){}^{L[Y][\der^{\delta}(X,\eta)]}} \\
			  & =\left(C(\aaa)^{\gamma}\right)^{\left.C(\aaa)\right.{}^{L[Y]}} \\
			  & =\left(C(\aaa)^{1+\gamma}\right)^{L[Y]}.
		\end{align*}
	\end{remark}

	The definition is as follows:
	\begin{enumerate}
		\item $\Der^{0}(X,\eta):=\Der(X,\kappa_{\eta},\vec{S}_{\eta})$.
		\item If $\delta=\alpha+\beta$ for $\alpha,\beta<\delta$, and $\beta$
		      is smallest such that this holds, then set
		      \[
			      \Der^{\delta}(X,\eta):=\Der^{\beta'}(X,\eta+1+\alpha)*\Der^{\alpha}(\der^{\beta'}(X,\eta+1+\alpha),\eta)
		      \]
		      where $\beta'=\beta-1$ if $\beta<\omega$ and otherwise $\beta'=\beta$.
		\item Otherwise, we can find in $\delta$ an increasing sequence of ordinals
		      $\left\langle \eta_{\alpha}\mid\alpha<\delta\right\rangle $
		      such that for every $\alpha<\delta$, $[\eta_{\alpha},\eta_{\alpha+1})$
		      has order-type $\alpha$. Then
		      \[
			      \Der^{\delta}(X,\eta):=\prod_{\alpha<\delta}\Der^{\alpha}(X,\eta+\eta_{\alpha}).
		      \]
	\end{enumerate}
	Let's see that this works. The initial step is clear. Assume $\delta=\alpha+\beta$.
	Then
	\[
		L[Y][\der^{\delta}(X,\eta)]=L[Y][\der^{\beta'}(X,\eta+1+\alpha)][\der^{\alpha}(\der^{\beta'}(X,\eta+1+\alpha),\eta)]
	\]
	where $\der^{\alpha}(\der^{\beta'}(X,\eta+1+\alpha),\eta)$ is generic over $L[Y][\der^{\beta'}(X,\eta+1+\alpha)]$ and the assumptions in
	\textbf{IH} hold,
	so $L[Y][\der^{\delta}(X,\eta)]$ satisfies that the $C(\aaa)$ sequence has length at least $\alpha+1$, and $C(\aaa)^{\alpha+1}=C(\aaa)^{L[Y][\der^{\beta'}(X,\eta+1+\alpha)]}$.
	Now again by \textbf{IH} the model $L[Y][\der^{\beta'}(X,\eta+1+\alpha)]$ has a $C(\aaa)$-sequence of length at least
	$\beta'+1$, with the $\beta'+1$ stage being $C(\aaa)^{L[Y]}$.
	Together we get a sequence of length at least $\alpha+1+\beta'+1=\alpha+\beta+1$
	(if $\beta$ is finite then $1+\beta'=1+\beta-1=\beta$ and if it
	is infinite then $1+\beta'=1+\beta=\beta$), and $\left.C(\aaa)^{\alpha+\beta+1}\right.^{L[Y][\der^{\delta}(X,\eta)]}=C(\aaa)^{L[Y]}$.

	Now, if $\Der^{\beta'}(X,\eta+\alpha)$ uses cardinals from
	\[
		\left\langle \kappa_{\gamma}\mid\gamma\in[\eta+1+\alpha,\eta+1+\alpha+1+\beta')\right\rangle
	\]
	and $\Der^{\alpha}(\der^{\beta'}(X,\eta),\eta)$  from
	\[
		\left\langle \kappa_{\gamma}\mid\gamma\in[\eta,\eta+1+\alpha)\right\rangle
	\]
	then $\Der^{\delta}(X,\eta)$ uses cardinals from
	\[
		\left\langle \kappa_{\gamma}\mid\gamma\in[\eta,\eta+1+\delta)\right\rangle .
	\]

	Consider now the last case (note that if $\delta=\omega$, this is
	exactly the construction in Theorem \ref{thm:descending-omega}).
	For any $\alpha,\beta<\delta$, $\alpha\ne\beta$, the forcings $\Der^{\alpha}(X,\eta+\eta_{\alpha})$
	and $\Der^{\beta}(X,\eta+\eta_{\beta})$ are independent of one another,
	that is we can use the product lemma, and the generic of $\Der^{\delta}(X,\eta)$
	is a disjoint union of generics for $\prod_{\beta<\alpha}\Der^{\beta}(X,\eta+\eta_{\beta})$,
	$\Der^{\alpha}(X,\eta+\eta_{\gamma})$ and $\prod_{\alpha<\beta<\delta}\Der^{\beta}(X,\eta+\eta_{\beta})$,
	which are mutually generic. So we have
	\[
		\left(C(\aaa)^{\alpha+1}\right)^{L[Y][\der^{\delta}(X,\eta)]}=\left(C(\aaa)^{\alpha+1}\right)^{L[Y][\bigcup_{\alpha\leq\beta<\delta}\der^{\beta}(X,\eta+\eta_{\beta})][\bigcup_{\beta<\alpha}\der^{\beta}(X,\eta+\eta_{\beta})]}
	\]

	Also note that the forcing $\prod_{\alpha<\beta<\delta}\Der^{\beta}(X,\eta+\eta_{\beta})$
	has the required properties to ensure that cardinalities and stationarity
	in the segment $\left\langle \kappa_{\gamma}\mid\gamma<\eta+1+\eta_{\alpha}\right\rangle $
	are preserved, so we can apply \textbf{IH}, and specifically Remark \ref{rem:larger iter},
	and inductively get that all the codings $\bigcup_{\beta\leq\alpha}\der^{\beta}(X,\eta+\eta_{\beta})$
	simply drop, i.e. that
	\begin{align*}
		 & \left(C(\aaa)^{\alpha+1}\right)^{L[Y][\bigcup_{\alpha\leq\beta<\delta}\der^{\beta}(X,\eta+\eta_{\beta})][\bigcup_{\beta<\alpha}\der^{\beta}(X,\eta+\eta_{\beta})]} \\
		 & =\left(C(\aaa)^{\alpha+1}\right)^{L[Y][\bigcup_{\alpha\leq\beta<\delta}\der^{\beta}(X,\eta+\eta_{\beta})]}                                                         \\
		 & =\left(C(\aaa)^{\alpha+1}\right)^{L[Y][\bigcup_{\alpha<\beta<\delta}\der^{\beta}(X,\eta+\eta_{\beta})][\der^{\alpha}(X,\eta+\eta_{\alpha})]}                       \\
		 & =\left.C(\aaa)\right.^{L[Y][\bigcup_{\alpha<\beta<\delta}\der^{\beta}(X,\eta+\eta_{\beta})]}
	\end{align*}

	so together we have
	\begin{align*}
		\left(C(\aaa)^{\alpha+1}\right)^{L[Y][\der^{\delta}(X,\eta)]} & =\left.C(\aaa)\right.^{L[Y][\bigcup_{\alpha<\beta<\delta}\der^{\beta}(X,\eta+\eta_{\beta})]}.
	\end{align*}

	This means that for every $\alpha<\delta$, we have a descending sequence
	of iterated $C(\aaa)$ of length at least $\alpha$, so we get a descending
	sequence of length $\delta$, which ends in
	\[
		\left(C(\aaa)^{\delta}\right)^{L[Y][\der^{\delta}(X,\eta)]}=\bigcap_{\alpha<\delta}\left.C(\aaa)\right.^{L[Y][\bigcup_{\alpha<\beta<\delta}\der^{\beta}(X,\eta+\eta_{\beta})]}.
	\]
	Recall that we are forcing with an club shooting iteration of length
	at most $\theta$ where $\delta<\theta<\kappa_{0}$, so by Theorem
	\ref{Thm:it-club-shooting-full}.\ref*{enu:distributivity-full} we
	work in a $\delta$-distributive forcing, so we can use the same methods
	as before to get that
	\[
		\left(C(\aaa)^{\delta}\right)^{L[Y][\der^{\delta}(X,\eta)]}=L[Y][\meq{\der^{\delta}(X,\eta)}]
	\]
	where $\sim$ is the equivalence relation derived from the $\delta$-normal
	system given by the projections
	\[
		\prod_{\alpha<\beta<\delta}\Der^{\beta}(X,\eta+\eta_{\beta})\to\prod_{\alpha'<\beta<\delta}\Der^{\beta}(X,\eta+\eta_{\beta})
	\]
	for $\alpha<\alpha'$. So as in the last part of the proof of Theorem
	\ref{thm:descending-omega}, we can prove that the stages of construction
	of $\left.C(\aaa)\right.^{L[Y]}$ and of $\left.C(\aaa)\right.^{L[Y][\meq{\der^{\delta}(X,\eta)}]}$
	are the same, so we get
	\begin{align*}
		\left.C(\aaa)^{\delta+1}\right.^{L[Y][\der^{\delta}(X,\eta)]} & =\left.C(\aaa)\right.^{\left(C(\aaa)^{\delta}\right)^{L[Y][\der^{\delta}(X,\eta)]}} \\
		                                                              & =\left.C(\aaa)\right.^{L[Y][\meq{\der^{\delta}(X,\eta)}]}                           \\
		                                                              & =\left.C(\aaa)\right.^{L[Y]}.
	\end{align*}

	This concludes the construction of the coding forcings $\Der^{\delta}(X,\eta)$.
	Now we get that in the model $L[\der^{\delta}(\der(A,0),1)]$ the
	$C(\aaa)$ sequence is of length at least $\delta+1$ and
	\begin{align*}
		\left.C(\aaa)^{\delta+1}\right.^{L[\der^{\delta}(\der(A,0),1)]} & =\left.C(\aaa)\right.^{L[\der(A,0)]}=L[A]
	\end{align*}
	as required.
\end{proof}

\section{Open questions}

The first obvious question is whether the last result can be pushed
to obtain an $\ord$ length iterated $C(\aaa)$ sequence. The first
obstacle to this is distributivity -- recall that in order to get
distributivity we need to have iterations which are shorter than the
first cardinal used for the coding. So to get longer iterations we'll
need to choose larger and larger coding cardinals. This might be possible,
but it is not straightforward. Additionally we'd need an analysis
of class length iterations for this type of forcing, and this is beyond
the scope of this paper. So the following is still open:
\begin{question}
	Is it consistent to have a model with an $\ord$ length $C(\aaa)$
	sequence?
\end{question}

A second questions arises from the fact that we force over $L$, and
use $L$'s $\square$-principle to get the mutually fat sets required
for the iteration. In section \ref{subsec:Forcing-non-reflecting-stationar}
we show that such sets can also be obtained by forcing, but then it
is not clear whether we can get those sets which we use as ``coding
tools'' into $C(\aaa)$ to begin with. This raises the question whether
large cardinals or failure of $\square$ might restrict the iterated
$C(\aaa)$ sequence. However, the results of \cite{Yaar-IteratingCofinalityoConstructible2023}
show that  in the case of $C^{*}$, the restriction of the length
of the $C^{*}$ sequence comes from \emph{lacking} large cardinals,
and large cardinals enable \emph{longer} sequences. In any case,
the following are open:
\begin{question}
	\begin{enumerate}
		\item    Can we force models with long $C(\aaa)$ sequence over any model of $\zfc$?
		\item Do large cardinals determine the possible length of such sequences?
		\item What is the length of the $C(\aaa)$ sequence in canonical models for large cardinals such as $L^{\mu}$?
	\end{enumerate}
\end{question}

In \cite{IMEL2} it is shown that under the assumption of a proper
class of Woodin cardinals the theory of $C(\aaa)$ is
set-forcing absolute. However, the length of the $C(\aaa)$ sequence
is \emph{prima facie} not a first-order statement, so it isn't clear
that this length will be preserved. However it is still worth investigating:
\begin{question}
	Does the assumption of a proper class of Woodin cardinals determine the length of the $C(\aaa)$ sequence?
\end{question}

A different kind of inquiry, in light of the results we mentioned
in the beginning about $\mathrm{HOD}$, is the following:
\begin{question}
	Is it consistent relative to $\zfc$ that the $C(\aaa)$ sequence is
	of length $\omega$, and either:
	\begin{itemize}
		\item $C(\aaa)^{\omega}\vDash \zf+\mathrm{\neg AC}$?
		\item $C(\aaa)^{\omega}\nvDash\zf$?
	\end{itemize}
\end{question}

A different line of inquiry stems from the observation that in our results we only used stationary subsets of ordinals in our coding schemes, rather than the more general notion of stationarity for sets of countable sets.
This suggests that the same results could be obtained for a model constructed using only the notion of stationarity for ordinals.
As the notion of stationarity for ordinals doesn't simply apply to general structures, defining this model in the form of $C(\LL)$ for some logic $\LL$ requires some care, and perhaps using the relativized hierarchy $L(A)$ for a predicate $A$ would be a more natural approach.
In any case, this is beyond the scope of this paper, and would be the focus of future research.

Finally, we have introduced the new notion of \emph{mutually fat set}s,
which we believe is worth further investigation in itself. The first
questions concern the difference between the various notions we introduced:
\begin{question}
	Let $K$ be an increasing sequence of regular uncountable cardinals,
	$\left|K\right|<\min K$, $\left\langle T_{\kappa}\mid\kappa\in K\right\rangle $
	a sequence with $T_{\kappa}\con\kappa$.
	\begin{enumerate}
		\item Assume $\left\langle T_{\kappa}\mid\kappa\in K\right\rangle $ are
		      mutually stationary and each $T_{\kappa}$ is fat. Does this imply
		      that the sequence $\left\langle T_{\kappa}\mid\kappa\in K\right\rangle $ is
		      \textbf{mutually} fat?
		\item Assume $\left\langle T_{\kappa}\mid\kappa\in K\right\rangle $ is
		      mutually fat. Does it imply that it is \textbf{strongly} mutually
		      fat?
	\end{enumerate}
\end{question}

We have used $\square$-sequences and forcing non-reflecting stationary
sets to obtain mutually fat sequences. It is worth investigating what
other methods are there for obtaining such sequences.

\section*{Acknowledgments}
I would like to thank my supervisor, Prof. Menachem Magidor, for his guidance and support without which this work would not have been possible.
I would also like to thank the anonymous referees of my PhD thesis and this paper for their helpful comments, corrections and suggestions.

Revision of this paper was done with the support of the European Research Council (ERC) under the European Union’s Horizon 2020 research and innovation programme (grant agreement No. 101020762).

\bibliographystyle{amsplain}
\bibliography{bibCaa}

\end{document}